\documentclass[notitlepage,12pt]{article}
\usepackage[novbox]{pdfsync}\usepackage{tabularx}
\usepackage[utf8]{inputenc}\usepackage{listings,verbatim}
\usepackage{graphicx}
\usepackage{amsmath,amsfonts,amssymb,amsthm,mathrsfs}
\usepackage{geometry}
\usepackage{ragged2e}
\usepackage{blindtext}
\usepackage{breakcites}
\graphicspath{{images/}}
\usepackage{caption}
\usepackage{subcaption}
\usepackage{url}
\usepackage{nccmath}
\usepackage{tikz}
\usepackage{float}
\newcommand{\blue}{\textcolor[rgb]{0.00,0.00,1.00}}
\newtheorem{teo}{Theorem}[section]
\newtheorem{defi}[teo]{Definition}
\newtheorem{cor}[teo]{Corollary}
\newtheorem{lem}[teo]{Lemma}
\newtheorem{eje}{Example}[section]
\newtheorem{rem}{Remark}[section]
\def\eeD{\end{defi}} \def\beD{\begin{defi}}
\def\beR{\begin{rem}} \def\eeR{\end{rem}}
\def\beL{\begin{lem}} \def\eeL{\end{lem}}
\def\beC{\begin{cor}
  }\def\eeC{\end{cor}}
  \def\beT{\begin{teo}}\def\eeT{\end{teo}}
\def\beXa{\begin{eje}} \def\eeXa{\end{eje}}
\def\BEN{\begin{enumerate}}
\def\EEN{\end{enumerate}}    \def\im{\item} \def\Lra{\Longrightarrow}
\def\bc{\begin{cases}
  }
\def\ec{\end{cases}}
   \def\I{\infty} \def\Eq{\Leftrightarrow}
  \newcommand{\beq}{\begin{eqnarray}
    }
\def\eeq{\end{eqnarray}}
\def\I{\infty}\def\T{\widetilde}
\def\fr{\frac}\def\DFE{disease free equilibrium}
  \newcommand{\be}[1]{$\begin{equation}\label{#1}$}
\newcommand{\ee}{\end{equation}}
\def\bea{\begin{eqnarray*}}\def\ssec{\subsection}
  \def\prf{{\bf Proof:} } 
\def\eea{\end{eqnarray*}} \def\lab{\label}\def\fe{for example }
 \def\eqr{\eqref}

\newcommand{\R}{\mathbb{R}}

\def\al{\alpha}
\def\ga{\gamma} \def\g{\gamma}  \def\de{\delta}  \def\b{\beta}
   
\def\ep{\epsilon} \def\La{\Lambda} 
 \def\la {\lambda} \def\La {\Lambda}
 \def\QED{\hfill {$\square$}\goodbreak \medskip}
\long\def\symbolfootnote[#1]#2{
\begingroup
\def\thefootnote{\fnsymbol{footnote}}\footnote[#1]{#2}
\endgroup}
\def\fn{\symbolfootnote}
\providecommand{\pp}[1]{\left[#1\right]} 
\providecommand{\pr}[1]{\left(#1\right)} 
\def\bep{\begin{pmatrix}} \def\eep{\end{pmatrix}}
\def\bev{\begin{vmatrix}} \def\eev{\end{vmatrix}}
\newcommand{\red}{\textcolor[rgb]{1.00,0.00,0.00}}

\def\sssec{\subsubsection}
\newcommand{\figu}[3]{
\begin{figure}[H]
\centering
\includegraphics[scale=#3]{#1}
\caption{#2\label{f:#1}}
\end{figure}
}
\def\wrt{with respect to }
\begin{document}

\title{On a three-dimensional and two four-dimensional oncolytic viro-therapy  models  
}

\author{Rim Adenane$^\text{a}$, \and Eric Avila-Vales$^\text{b}$, \and Florin Avram$^\text{c}$, \and Andrei Halanay$^\text{d}$, \and Angel G. C. P\'erez$^\text{b}$ 
}
\maketitle

\begin{center}
	{\small
		$^\text{a}$
		Département des Mathématiques, Université Ibn-Tofail, Kenitra, 14000,
 Maroc\\
		
		$^\text{b}$
		Facultad de Matem\'aticas, Universidad Aut\'onoma de Yucat\'an, Anillo Perif\'erico Norte, Tablaje Catastral 13615, C.P. 97119, M\'erida, Yucat\'an, Mexico\\
		
		$^\text{c}$
		Laboratoire de Math\'{e}matiques Appliqu\'{e}es, Universit\'{e} de Pau, 64000, Pau,
 France\\
		
		$^\text{d}$
		Department of Mathematics and Informatics,  Polytechnic University of Bucharest, 062203, Bucharest, Romania
	}
\end{center}
\bigskip

\begin{abstract}
We revisit here and carry out further works on tumor-virotherapy compartmental models of \cite{Tian11,Tian13,Tian17,Guo}. The results of these papers are only slightly pushed further. However, what is new is the fact that we make public our electronic
notebooks, since we believe that easy electronic reproducibility is crucial in an era in which the role of the software becomes very important.
\end{abstract}

\tableofcontents
\section{Introduction}

{\bf Compartmental models} became famous first in mathematical epidemiology, following the pioneering work of Kermack and McKendrick \cite{KeMcK} on the SIR model; see \cite{Haddad} for  other domains of application, and  for some general theory. In  the last thirty years, they have penetrated also in mathematical virology \cite{Perel,Nowak,WodKom,Bocharov}, and in mathematical oncolytic virotherapy, i.e. in the modeling of the use of viruses for treating tumors \cite{santiago2017fighting,Rockne,pooladvand2021mathematical}.

 We may distinguish between at least two main directions of work in these fields.
 \BEN \im Part of the literature is dedicated to creating models to fit specific
viruses and therapies -- see \fe\ \cite{Perelson,Perelson02,antonio2002oncolytic,smith2003virus,Wod,Pillis,Dalal,Tuckwell,yuan2011stochastic,YuWei,
huang2011complex,chenar2018mathematical}.  The models proposed are   high dimensional, and hence only  analyzable numerically, for particular instances of the parameters.

\im Another part, which is our concern here, is in applying sophisticated mathematical tools, notably the theory of bifurcations for dynamical systems, to ``lower dimensional caricatures" of more complex models.
This  requires  the use of both symbolic software like Mathematica, Maple, or Sagemath, and also of sophisticated numeric continuation and bifurcation packages like MatCont (written in Matlab), PyDSTool (Python),  XPPAuto (C)  -- see \cite{blyth2020tutorial} for a recent review,
and   BifurcationsKit (written in Julia).
\EEN
  In our work below,  we have combined the use of MatCont -- see \cite{ghub2022fourdim}
   with  that of Mathematica -- see \cite{codeM}, and in particular the   package  EcoEvo.
   The notebooks offered on GitHub are an important part of our work, and we attempted to achieve a roughly  one to one correspondence between the equations numbered  in the text and those displayed in Mathematica.

{The origins of the glioma viro-therapy four-compartment $(x,y,v,z)$ model considered here, where  untreated and tumor cells are denoted respectively by $x,y$,  virus cells by $v$, and innate immune cell by $z$, are in \cite{o1999fas,Fried}.\fn[5]{A four-dimensional model considerably more complex was proposed in \cite{senekal2021natural}.}
Interestingly, these papers suggested  a density dependent rate of immune cells, linear up to a threshold $z_0$, and quadratic afterwards.
 ``The first process occurs when $z$ is small and yields a linear clearance; the
second process occurs when $z$ is large and yields a quadratic clearance" \cite[pg 2]{Fried}.} Subsequent papers of Tian  \cite{Tian17},
\cite{Guo}  tackled symbolically the two particular cases $z_0=\I$, $z_0=0$.
For further developments and further outstanding questions in the field, see
\cite{wei21bis,Tian22I,Tian22II}.

Since the quadraticity is hard to ascertain, we propose to study  a unification of  the four-compartment systems studied in \cite{Tian17,Guo}:
\begin{equation} \label{viroeq}
	\begin{aligned}
		\frac{dx}{dt} & = \la x\left(1-\frac{x+y}{K}\right) -\beta xv \\ 
		\frac{dy}{dt} & = \beta xv-\ga   y-\b_y yz                    \\
		\frac{dv}{dt} & = b\ga   y -\beta xv -\de   v -\b_v  vz       \\
		\frac{dz}{dt} & = z(\rho \b_y y-cz^\ep),\ \ep \in \{0,1\},
	\end{aligned}
\end{equation}
where $x$, $y$, $v$ and $z$ represent the populations of uninfected (untreated) tumor cell population, infected tumor cell population, free virus and innate immune cells, respectively.

\beR The invariance of the first quadrant (also called ``essential non-negativity") is immediate since each component $f_i(X)$ of the dynamics  may be decomposed
   as \bea f_i(X)=g_i(X)-x_i h_i(X),\eea
  where $g_i,  h_i$ are polynomials with nonnegative coefficients,  and $x_i$ is the variable whose rate is  given by $f_i(X)$. In fact, under this absence of ``negative cross-efects",  
    even more is true:  the model admits a ``mass-action representation" by the so-called   ``Hungarian lemma" \cite{hars1981inverse,Haddad}, \cite[Thm. 6.27]{Toth}\fn[4]{The previous virology  literature does not seem  to be aware of this result, and offers  direct proofs instead.}  
\eeR

\beR Scaling all the variables by $x= K \T x$, $y= K \T y$, ... has the effect of multiplying all the quadratic terms by $K$, and  one may finally assume $K=1$, at the price of renaming some other parameters. Also, scaling time by a constant allows choosing another parameter as $1$. Below, we will follow occasionally \cite{Tian11,Tian17} in choosing $K=\g=1$, which simplifies a bit the results. \eeR
Figure \ref{f:model_diagram} depicts a schematic diagram of this model. The interpretation of parameters can be seen in Table \ref{tab2}.


\figu{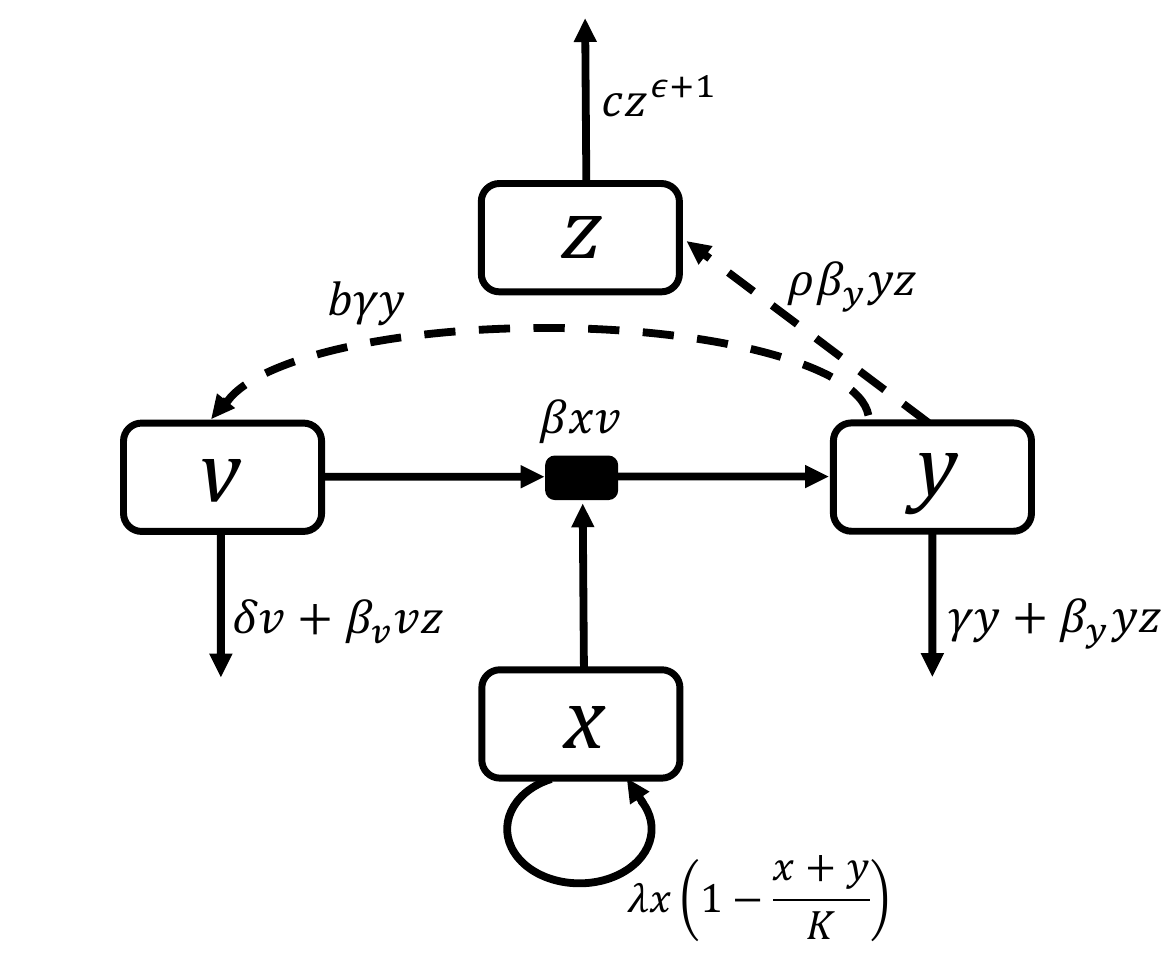}{Schematic diagram of model \eqref{viroeq}. The compartments $x$, $y$, $v$ and $z$ denote uninfected tumor cells, infected tumor cells, free virus and innate immune cells, respectively. Continuous lines represent transfer between compartments. Dashed lines represent viral production or activation of immune cells.}{1}

\begin{table}[H]
    \caption{Interpretation of parameters for model \eqref{viroeq}.}\label{tab2}
    \centering
    \begin{tabular}{|l | l|}
    	\hline
    	Symbol & Description                                                                \\ \hline
    	$\la$       & intrinsic growth rate of uninfected tumor cells \\ \hline
    	$K>0$       & carrying capacity of uninfected tumor cells \\ \hline
    	$\beta >0$       & viral infection rate                        \\ \hline
    	$\b_y$       & rate at which immune system removes infected tumor cells \\ \hline
    	$\ga  >0 $  & lysis rate of infected cells          \\ \hline
    	$b {\ge 1}$  &  virus burst size         \\ \hline
    	$\de  $  &  clearance rate of viruses         \\ \hline
    	$\b_v $  &   rate at which immune system removes viruses        \\ \hline
    	$\rho \b_y:=\b_z >0$  &  proliferation rate of immune cells due to the interaction with infected tumor cells          \\ \hline
    	$c$  &  rate of clearance of immune cells \\ \hline
    \end{tabular}
\end{table}

\beR When $\b_y=\b_v =0$ (the immune system is totally inefficient), the general  model \eqr{viroeq} reduces to a three-compartment  viral model \eqr{Tian11} of \cite{Tian11,Kim}.

The fourth compartment for the immune system was subsequently modeled differently in \cite{Tian17} and in \cite{Guo} (where $K=\infty$). We have unified these two papers by adding the parameter $\ep$, which equals $0$ in \cite{Tian17} and $1$ in \cite{Guo}. \eeR

The \cite{Tian11} three compartment model has been analyzed symbolically up to a point, and it ended with  {a list of open problems,   which awoke our attention, since they seemed to be still open.} Other interesting open problems were raised by the two four-compartment model(for example, the local stability of certain points was only established  in particular cases, numerically).

 We point out now another important open problem, not mentioned in \cite{Tian11}.

 {\bf Q:} Can chaos arise in model \eqr{viroeq}?

 Note that while we do have the right  to hope for the absence of  complicated dynamical behaviors, since we are dealing with a pseudo-linear, essentially non-negative  system,  this is by no means guaranteed. Indeed, complicated dynamics like multiple ``concentric" cycles have been found in \cite{RuanWang}, and  in the parallel ecology literature on {\bf three dimensional food-chains},  chaos is known to occur as well \cite{klebanoff1994chaos,kuznetsov1996remarks,
 kuznetsov2001belyakov,deng2001food,deng2003food,deng2004food,deng2006equilibriumizing,
 deng2017numerical}.

While very interesting and worthy of further investigation,   the virology papers cited above suffer from the lack of providing supporting electronic notebooks. The importance of symbolic and numeric computing in mathematical biology cannot be overstated (see for example \cite{brown2006algorithmic}). 

{\bf Electronic reproducibility}.
As  emphasized already 30 years ago,  the  opportunity we have nowadays    of being able to accompany our pencil calculations with electronic notebooks
``gives  a new meaning to reproducible research" \cite{Repro}. Following efforts of numerous people, for example \cite{buckheit1995wavelab,Donoho},  lots of progress has been achieved, as witnessed by the existence of the platform GitHub. Unfortunately, the percentage of researchers who take the time to tidy their notebooks and make them available on GitHub is still infinitesimal in some fields.

Our main contribution below is in  providing electronic notebooks, where the readers may recover the results of the
previous works of \cite{Tian11,Tian13,Tian17,Guo},  and then modify them as
they please, for analyzing similar models. Note this is a non-trivial task, and it goes in a direction orthogonal to that of most of the current literature. 

{\bf Contents}. We start by revisiting in Section \ref{s:Tian11} the three-dimensional model of \cite{Tian11}, which had been already
 essentially  solved  symbolically. However, with  help from Mathematica, 
 we resolve one of the  problems left open in
 \cite{Tian11}.

In Section
\ref{s:mod} we  introduce   a ``generalized virus" model, geared at unifying previous studies and initiating new directions of research
(as typical in the field, we will not be able to answer all our questions).

Some first results for our general model are then presented
in Sections \ref{s:pob}, \ref{s:bou}.

The particular case of  \cite{Tian17} is revisited in Section \ref{s:Tian17}.

We turn then to the complete viro-therapy and immunity model with logistic growth in Section
\ref{s:Eric}.


\section{Warm-up: the  3 dimensional viral model  \cite{Tian11,Kim}\lab{s:Tian11}}

The three-dimensional tumor-virus model proposed in \cite[(5)]{Tian11}, \cite{Tian13} is:
\begin{equation}\lab{Tian11}
	\begin{aligned}
		\frac{dx}{dt} & =\la x\left(1-\frac{x+y}{K}\right)-\beta xv                             \\
		\frac{dy}{dt} & =\beta xv-\ga y                                                         \\
		\frac{dv}{dt} & =b\ga  y-\beta xv-\de  v, \; {x+y\leq K,\ v\leq \frac{b\ga   K}{\de }}.
	\end{aligned}
\end{equation}

{\bf Brief history}.
A similar three-dimensional $(x,y,v)$ model, with linear growth, and with the term $\beta xv$ present in all the equations
seems to have been first  proposed by Anderson, May and Gupta  \cite{anderson1989non}, as a model for   the interaction of   parasites with host-cells, in particular red blood cells (RBC). Subsequently, this became known as the Novak-May model \cite{Nowak}, and has
 been applied in many other directions, for example  by Tuckwell \& Wan \cite{Tuckwell}, as a model for HIV-1 dynamics. See also  \cite{Tian11}, \cite{Tian13} for a version with delay, \cite{TuanTian} for a stochastic version, , and see \cite{Camara}   for a stochastic version with    ``saturated infection rate" in which $\b$ is replaced by $\b(x,y)=\fr{\b}{x+y+\al}$.

Factorization yields easily the three equilibrium points for the model \eqr{Tian11}. The first two $E_0=(0,0,0)$, $E_{K}=(K,0,0)$ are ``infection free", and the third equilibrium
$$ E_*=\left(\frac{\de }{\beta (b-1)},\ 
\la \fr{\de  }{\b(b-1)}
\frac{K\beta (b-1)-\de }{K\beta \ga (b-1)+\la\de },\ \la \fr{\ga }{\b}\frac{K\beta (b-1)-\de}{K\beta\ga  (b-1)+\la\de }\right)$$
 is interior to the domain.

The explicit eigenvalues of the Jacobian at the first  equilibrium point make it  a saddle point. Similarly, examining the Jacobian shows that the second point is stable when
 \beq \label{R0} R_0:=\frac{\beta K}{\beta K +\de }b\eeq
 is smaller than $1$ and unstable when $R_0>1$, where
$R_0$ is  the famous   ``basic reproduction number".  Note however that the results of
\cite{Van, Van08} do not apply here due to the existence of two ``disease free" equilibria.

The critical
$b$ which makes $R_0=1$ is
\beq \lab{bcr} b_0=1+\fr{\de}{\b K},\eeq
confirming Lemma \cite[Lem. 3.4]{Tian11}.
Writing the third point  as
\begin{equation*}
E_*=\left(\frac{\de }{\beta (b-1)},\ \frac{\de  \la}{\b(b-1)}\fr{\beta K +\de} {K\beta\ga  (b-1)+\la\de }
(R_0-1 ),\ \fr{\ga   \la}{\b}\frac{\beta K +\de }{K\beta\ga  (b-1)+\la\de } (R_0-1 )\right).
\end{equation*}

It is convenient   to rescale time by $\ga$ and the variables by $K$, the net result being that these variables may be assumed to equal $1$ \cite[Sec. 3.1]{Tian11}. The third point simplifies then to
\begin{equation}
E_*=\left(\frac{\de }{\beta (b-1)},\ \frac{\de  \la}{\b(b-1)}\fr{\beta  +\de} {\beta  (b-1)+\la\de }
(R_0-1 ),\ \fr{   \la}{\b}\frac{\beta  +\de }{\beta  (b-1)+\la\de } (R_0-1 )\right)
\end{equation}
and we see that this point enters  the nonnegative domain precisely when $R_0=1$, at $E_*=E_*(b_0)=E_K=(1,0,0)$.


 $E_*$  interior to the invariant domain iff $R_0>1$. Its stability   may be tackled  via the Routh-Hurwitz
 conditions, which,
 {at order three,  amounts to $\bc Tr(J)<0,\\
     Tr(J) M_2(J)<Det(J)<0,\ec$
where $M_2$ is the sum of the second-order principal leading minors of the Jacobian matrix $J$ at $E_*$.}
 Now the first and last inequalities are always satisfied in our case \cite[Thm. 3.7]{Tian11}
 {since,
 $\bc Tr(J):=-1+\frac{\delta(b \beta+\lambda)}{\beta(1-b)}<0\\ Det(J):=(1-R_0)\pr{\frac{\delta \lambda(\beta+\delta)}{\beta(b-1)}}<0\ec$}, and thus the local stability of the point $E_*$ holds iff
\beq \lab{H} H(b):= Det(J)-Tr(J) M_2(J)=-a_3+ a_1 a_2>0,\eeq
where

\begin{align*}
	a_1 & := \frac{\beta(b+b\delta-1) + \delta\lambda}{(b-1)\beta},                                                                                                                                                  \\
	a_2 & := \frac{\delta\lambda\big( (b-1)\beta\left(\beta-1+\delta + b(1-\beta+\delta)\right) + \left((b-1)^2 \beta + b\delta^2 \right) \lambda \big)}{ (b-1)^2 \beta \left( (b-1)\beta + \delta\lambda \right) }, \\
	a_3 & := \delta\lambda\left( 1 + \frac{\delta}{\beta(1-b)} \right).
\end{align*}

\beR As a check, note that $H(b_0)=\lambda (1 +\delta +\beta) (1 +\delta +\lambda +\beta )>0$, and so $E_*$ is stable at the critical point  when $E_{K}$ loses its stability,  as expected.
\eeR

To analyze the sign of $H(b)$,  we note first  that
 its denominator  $(b-1)^3 \beta^2 ((b-1) \beta + \delta \lambda )$ is always positive ({see second cell in  notebook \cite{codeM3}}).

Positivity reduces thus to the positivity of the numerator, which is a fourth order polynomial
\beq \lab{Phi} \Phi(b)=B_4 b^4+B_3 b^3+ B_2 b^2+ B_1 b+B_0,\eeq
and may be investigated via Descartes's rule.

 The  coefficients are
 \begin{eqnarray*}
\bc
B_4:=-\beta ^3, \\
B_3:=\beta ^2 (-\beta  (\delta -3)+\delta  (\delta +3)+\lambda +1), \\
B_2:=\beta  \left(\beta ^2 (2 \delta -3)-3 \beta  (2 \delta +\lambda +1)+\delta  \lambda  (\delta  (\delta +3)+\lambda +1)\right), \\
B_1:=-\beta ^3 (\delta -1)+\beta ^2 \left(-\delta ^2+3 \delta +3 \lambda +3\right)-\beta  \delta  \lambda  (3 \delta +2 \lambda +2)+\delta ^3 \lambda ^2, \\
B_0:= \b (\lambda +1) (\delta  \lambda -\beta ).
\ec
\end{eqnarray*}

  By using $B_4<0,\Phi(b_0)=\frac{\delta ^3 (1 +\lambda ) (1+\beta  +\delta ) (1+\beta  +\delta +\lambda )}{\beta }>0$, \cite[Lem. 3.8]{Tian11} concludes that
 the  fourth order polynomial $\Phi(b)$ must have at least one root  larger  than $b_0$, and one root smaller than $b_0$. Letting $b_H$ denote the smaller root  larger than $b_0$, \cite[Thm 3.9]{Tian11} concludes that  local stability  holds in $(b_0,b_H)$. 
 Also,
 $b_H$ is a candidate for a Hopf bifurcation, by the following elementary Lemma.
 \beL \cite[Lem. 3.10]{Tian11}
 A cubic polynomial $\la^3+ a_{1} \la^{2}+ a_{2} \la+  a_3$ with real coefficients
has a pair of pure imaginary roots if and only if $a_2 > 0$ and $a_3 = a_1 a_2$. When it has
pure imaginary roots, these are given by $\pm i \sqrt{a_2}$, the real root is
given by $-a_1$, and $a_1 a_3 > 0$.
\eeL
\beR  Higher dimension extensions exist as well -- see \cite{farkas1992use,guckenheimer1997computing}. \eeR

\cite{Tian11}  conjectured that $E_*$ may regain its stability at still larger values of $b$, after crossing yet larger roots, and the
 question of whether this may occur: ``What conditions can guarantee that the function $H(b)$ has four, three, and two distinct real zeros?"

 The precise classification of polynomials by their number of roots is a complicated problem \cite{Prod}, and we do not address it below.
 We may answer however the stability question, using the observation in the next remark.

\beR {The real roots smaller than $b_0$ have no importance  (for stability)}, so the real question is whether the fourth-order polynomial $\Phi(b)$ given by \eqr{Phi} may have {\bf more than one
real root larger than} $b_0$.

This can be tackled  by shifting the polynomial to $\Phi(b_0+x)$ and applying Descartes upper bound on the maximum number of positive roots via the number of sign changes in the sequence of coefficients of the shifted polynomial.

\eeR

\beL \lab{l:bH} The polynomial $\Phi(b)$ defined in \eqr{Phi} has precisely one
real root $b_H$ larger than $b_0$.
\eeL

\begin{proof}
	The coefficients of the shifted polynomial $\Phi(b_0+x)$ are
	\bea \bc
	\tilde{B}_0 := \frac{\delta ^3 (\lambda +1) (\beta +\delta +1) (\beta +\delta +\lambda +1)}{\beta }, \\
	\tilde{B}_1 := \delta ^2 \left(\beta  (2 \delta  \lambda +3 \delta +3 \lambda +3)+(\delta +2) \lambda ^2+2 \delta  (\delta +3) \lambda +\delta  (3 \delta +5)+5 \lambda +3\right), \\
	\tilde{B}_2 := \beta  \delta  \left(-\beta ^2+\delta  (\delta +3) \lambda +3 \delta  (\delta +1)+\lambda ^2+4 \lambda +3\right), \\
	\tilde{B}_3 := \beta ^2 (-\beta  (\delta +1)+(\delta -1) \delta +\lambda +1), \\
	\tilde{B}_4 := -\beta ^3. \ec
	\eea
	The first two are positive and the last negative, so in order  to have three roots larger than $b_0$ it is necessary that the third coefficient is negative and the fourth positive.
	Now each of this inequalities admits solutions, but the command
$$Reduce[\{cofi[[4]]>0\&\&cofi[[3]]<0\}]$$
{at the end of the second cell}
in the Mathematica file \cite{codeM3} yields False, telling us that the system of the two inequalities doesn't.
Similarly,
$$FindInstance[\{cofi[[4]]>0\&\&cofi[[3]]<0\},par]$$
fails.
	The diligent reader is invited to provide a ``human proof", but warned  that this seems hard.
\end{proof}

A bifurcation diagram of model \eqref{Tian11} as $b$ varies is illustrated in Figure \ref{f:bif11T}.

\figu{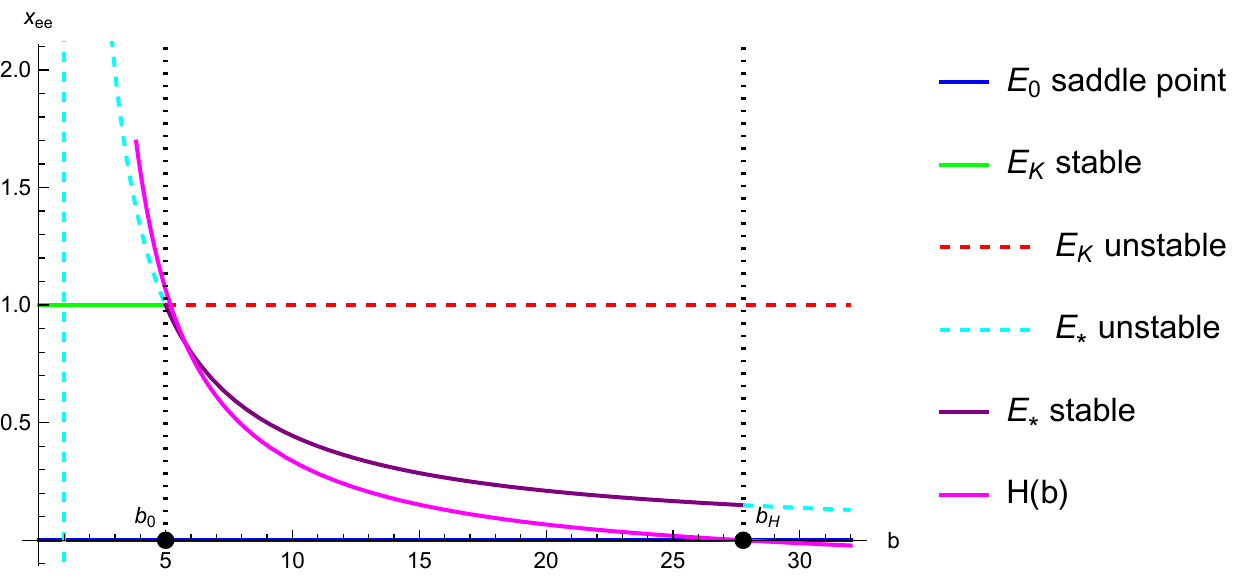}{Bifurcation diagram when $b$ varies, when $\lambda = 0.36$, $\beta = 0.11$, $\delta = 0.44$, $K=\gamma=1 \Lra b_0= 1+ \frac{\delta}{\beta}=5, b_H=27.7664$. When $b$ is bigger than the  Hopf bifurcation point $b_H$, there are no stable fix points.}{1}

Figures \ref{f:fig5T} and \ref{f:cy113D} show an illustration of the cycle arising  with the  parameter set above, at a value  $b=28$ slightly larger than   $b_H$, {see also \cite[Figure 7]{Tian13}}.

\begin{figure}[H]
    \centering
    \begin{subfigure}[a]{0.5\textwidth}
    \centering
        \includegraphics[width=\textwidth]{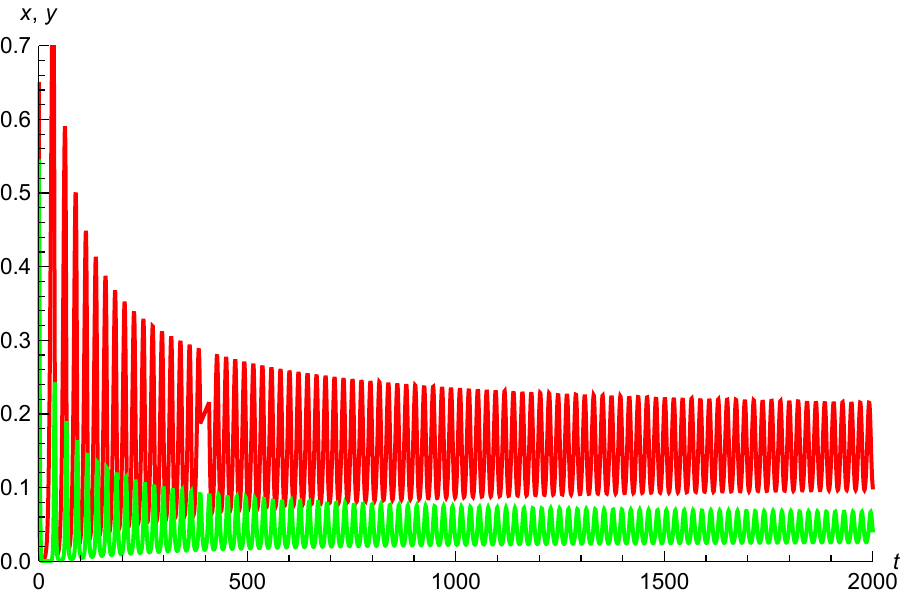}
        \caption{ $(x,y)$-time plot suggests the existence of periodicity.\label{f:fig5T}}
    \end{subfigure}%
    ~
     \begin{subfigure}[a]{0.6\textwidth}
     \centering
        \includegraphics[width=\textwidth]{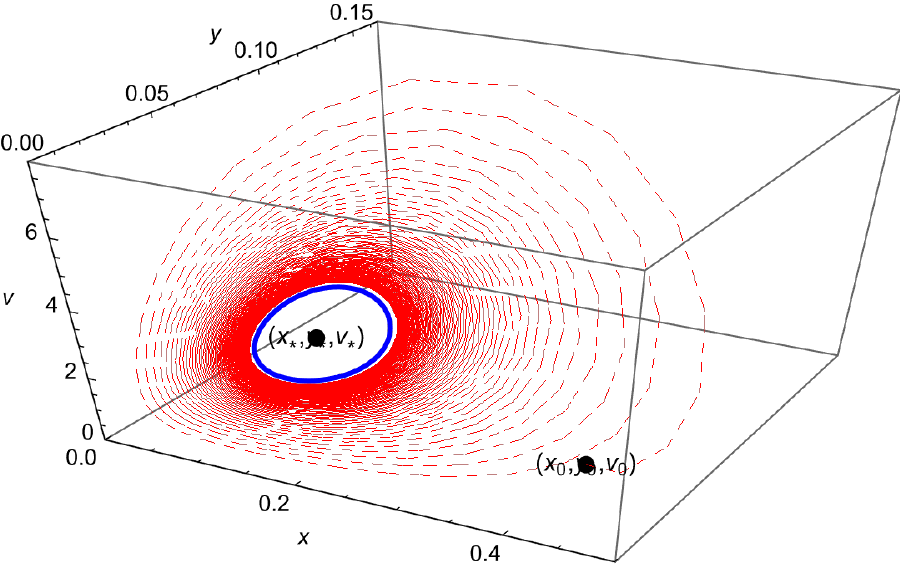}
        \caption{3D-parametric plot showing convergence towards an attracting cycle.  Two paths are displayed, one starting near the unstable fixed point  $E_*=(0.148148, 0.0431317,  2.64672),$ with eigenvalues $\{-1.51022,0.000296187\pm0.298909Im\}$, and one starting far away.\label{f:cy113D}}
    \end{subfigure}
    \caption{Time and 3D parametric plot  when $b=28>b_H=27.7664$.}
    \end{figure}

\section{The four-compartment viro-therapy and immunity model \eqr{viroeq}\lab{s:mod}}

\subsection{Boundedness \lab{s:pob}}

\begin{teo} 

 The epidemiological domain $$\Omega=\left\{(x,y,v,z)\in \mathbb{R}_+^4; x(t)+y(t)\leq K,\ v(t)\leq \frac{b\ga   K}{\de },\ z(t)\leq \zeta\right\},$$
  where
\[\zeta =
\begin{cases}
	\dfrac{\rho \beta b\ga K^2}{\delta\min\{\ga, c\}}, & \text{if } \ep=0; \\
	\frac{\rho \b_y K}{c},                                         & \text{if } \ep=1.
\end{cases}\]
is a positively invariant set.

 \end{teo}

\beR The first conditions on $x,y,v$ appear already in \cite[Lem. 1]{Tian17}.
\eeR

\begin{proof}
	 By adding the first two equations in \eqref{viroeq}, one obtains
	\begin{align*}
		\dot{x}+\dot{y} & = \la x\left(1-\frac{x+y}{K}\right)-\beta xv+\beta xv-\b_y yz-\ga   y \\
		                & \leq \la x \left(1-\frac{x+y}{K}\right).
	\end{align*}
	By a comparison argument we obtain that $\limsup_{t\to\infty}x(t)+y(t)\leq K$. This implies that for all $\varepsilon>0$ there is $t_1>0$ such that if $t>t_1$ then $x(t)\leq K+\varepsilon$ and $y(t)\leq K+\varepsilon$. Then, for $t>t_1$ we have
	$$\dot{v}=b\ga   y-\beta xv-\de  v-\b_v  zv\leq b\ga   (K+\varepsilon)-\de  v,$$
	from which we deduce that $\limsup_{t\to\infty}v(t)\leq \frac{b\gamma   K}{\de }$.
	\\ Lastly, for the boundedness of $z$, we will divide the proof in two cases.
	Let $\varepsilon_2>0$ and take $t_2>0$ such that
	$$x(t)+y(t) \leq K+\varepsilon_2
	\text{\quad and\quad}
	v(t) \leq \frac{b\gamma K}{\de}+\varepsilon_2
	\text{\quad for\quad}t>t_2.$$
	Consider first the case $\epsilon=0$. Let $w(t)=y(t)+\fr 1{\rho}z(t)$. Then, for $t>t_2$, we have
	\begin{align*}
		\dot{w} & = \dot{y}+\fr 1{\rho}\dot{z}                                                                                             \\
		        & = \b xv - \gamma y - \frac{c\b_y}{\b_z}z                                                                                      \\
		        & \le \b\left(K+\varepsilon_2\right)\left(\frac{b\gamma K}{\de}+\varepsilon_2\right) - \sigma\left(y+\fr 1{\rho}z\right),
	\end{align*}
	where $\sigma:=\min\{\ga, c\}$. It follows that
	\[\limsup_{t\to\infty} \fr 1{\rho}z(t)
	\leq \limsup_{t\to\infty} w(t)
	\leq \dfrac{\beta b\gamma K^2}{\delta\sigma},
	\]
	and finally
	$$\limsup_{t\to\infty}z(t) \le \dfrac{s\beta b\gamma K^2}{\b_y\delta\min\{\ga, c\}}.
	$$
	Lastly, in the case when $\epsilon=1$, we obtain
	\begin{align*}
		\dot{z} & = \b_z yz - cz^2                                                                            \\
		        & \leq \b_z\left(K+\varepsilon_2\right)z - cz^2                                              \\
		        & = \b_z\left(K+\varepsilon_2\right)z\left(1-\frac{c}{\b_z\left(K+\varepsilon_2\right)}z\right)
	\end{align*}
	for $t>t_2$. From this, we deduce that $\limsup_{t\to\infty}z(t)\leq \frac{\b_z K}{c}$.
\end{proof}

\subsection{Boundary equilibria and their stability
 \lab{s:bou}}
Factoring  the last and first equilibrium equations yields four points  which have either $z=0$ or $x=0$:
\begin{teo}
 The  fixed points with $z=0$ or $x=0$ are respectively:
 \begin{itemize}
\item $E_0=(0,0,0,0)$
\item $E_K=(K,0,0,0)$
\item \beq \label{Es} E_*=
\left(\frac{\de }{\beta (b-1)},
y_*:=\frac{\de  }{\beta (b-1)}\frac{b   \lambda  (R_0-1) }
{\lambda(b-  R_0) +(b-1) \gamma  R_0  }, \fr{\ga (b-1)}{\delta} y_*,0\right).\eeq
\item $E_N=\left(0, y_e=\fr c s, -y_e\frac{b  \gamma  \mu _y}{\gamma   \mu _v-\delta   \mu _y}, -\frac{\gamma }{\mu _y}\right)$ when $\ep=0$ and $(0,-y_e \frac{\ga  }{\mu},y_e\frac{\ga  ^2b}{\ga  \b_v -\de \mu},-\frac{\ga  }{\b_y}  )$ when $\ep=1$.  This is always outside the domain and will be ignored from now on.

\end{itemize}

\label{Viro_teo_413}

\end{teo}

\beR The first three fixed points appear already as solutions of the three-dimensional system \cite[(5)]{Tian11} obtained when the immune system is  inexistent.

Indeed, with $K=1$, $\ga=1$ (after rescaling), as in section \cite[Sec. 3.1]{Tian11}, the second and third point become $(1,0,0)$  and
$\left(\frac{\de }{\beta (b-1)},
\la\fr{\de  }{\b(b-1)}
\frac{\beta (b-1)-\de }{\beta  (b-1)+\la\de },\fr{\ga   \la}{\b}\frac{\beta (b-1)-\de}{\beta  (b-1)+\la\de }\right)$ --see \cite[Sec. 3.2]{Tian11}.

The fourth fixed point appeared also already, in the linear growth problem of \cite{Guo}.
\eeR

The Jacobian matrix of the system when $K=\ga=1$ is given by
$$J(x,y,v,z)=\left(
\begin{array}{cccc}
 \lambda -\frac{\lambda  (2 x+y)}{K}-\beta  v & -\frac{\lambda  x}{K} & -\beta  x & 0 \\
 \beta  v & -\gamma -z \mu _y & \beta  x & -y \mu _y \\
 -\beta v & b \gamma  & -\delta -z \mu _v+\beta  (-x) & -v \mu _v \\
 0 & s z & 0 & s y-c (\epsilon +1) z^{\epsilon } \\
\end{array}
\right).$$

At the boundary fixed points, the Jacobian has a block-diagonal form,  which simplifies the  stability analysis.

\begin{teo} $E_0$ is always a saddle point.\end{teo}
\begin{proof}
Since
$$J(E_0)=\left(\begin{matrix}
\lambda & 0 & 0 & 0\\
0 & -\ga   & 0 & 0\\
0 & b\ga   & -\de   & 0\\
0 & 0 & 0 & c_e
\end{matrix} \right),$$ where $c_e=c \left(-0^{\epsilon }\right) (\epsilon +1)=\bc -c &\ep=0\\0&\ep=1\ec$, it
 has always  one positive eigenvalue  $\lambda>0,$ and at least two negative eigenvalues  $-\ga , -\de$.
\end{proof}

\subsubsection{Stability of the boundary fixed point $E_K$ \label{s:Sta}}

Here we will prove  the existence of a stability transition of $E_K$ when $R_0=1$,
in the spirit of  the ``$R_0$ alternative". The  proof is standard when $\ep=0$\fn[4]{and global stability  holds as well under the assumptions $R_0<1$ and $y_e > 1$ \cite[Prop. 4]{Tian17}},
but the result is more delicate when $\ep=1$, since the Jacobian is singular at $E_K$:
$$\left(
\begin{array}{cccc}
 -\lambda  & -\lambda  & -\beta  & 0 \\
 0 & -z \mu _y-1 & \beta  & 0 \\
 0 & b & -\beta -\delta -z \mu _v & 0 \\
 0 & s z & 0 & -c (\epsilon +1) z^{\epsilon } \\
\end{array}
\right)$$

\begin{teo} When $\ep=1$, $E_K=(K,0,0,0)$ is
\BEN \im  unstable if $R_0>1 \Eq  b > b_0$;
\im  stable if $R_0<1,   \ep=0$;

\im when $R_0<1,   \ep=0$, the equilibrium $E_K$ is locally stable and
 local asymptotic stability holds
with respect to $(x,y,v)$, i.e. every solution that starts close enough to $E_K$
satisfies $\lim\limits_{t\to\infty} (x, y, v)(t)=(K,0,0)$.
\EEN
\end{teo}

\begin{proof}
The Jacobian at $E_K$ is $$ J(E_K)= \left(
\begin{array}{cccc}
 -\lambda  & -\lambda  & -K \beta  & 0 \\
 0 & -\gamma  & \beta  K & 0 \\
 0 & b \gamma  & -K \beta  -\delta  & 0 \\
 0 & 0 & 0 & 0 \\
\end{array}
\right).$$

The  block diagonal structure puts in evidence  an   upper $1\times1$ block with  negative eigenvalue  $-\lambda$, and a  lower $1\times1$ block with   eigenvalue equal to $0$. The remaining  middle $2\times2$  diagonal block  has  determinant   $\ga(\beta K +\de)(1-R_0)$, implying  eigenvalues of different sign when
$R_0>1$, yielding the first part of the result.

\im The trace is always negative, implying  two negative eigenvalues from the middle block when
$R_0<1$.
Together with the last eigenvalue $-c$, this yields  the second result.

\im When
$R_0<1, \ep=1$, our system  is  in the delicate situation
covered by the Lyapunov-Malkin Theorem (since the fixed point is never hyperbolic and the  Hartman-Grobman Theorem does not apply),  \cite[Ch. IV, \S 34]{Malkin}, \cite{Marsden}. Even if one condition of this theorem is not fulfilled for the shifted vector $(x_1=x-K,y,v,z)$, namely $f_4 (0,0,0,z)=-c z^2$ is not zero for every $z$ as required , the proof still yields   simple local stability due  to the specific form of $f_4 (x_1, y, v, z)= z (s  y-c z^2)$. The asymptotic behavior of $z$ cannot be inferred  using the original proof any more.

\end{proof}

\subsubsection{Stability of the boundary fixed point $E_*$ \label{s:Sta}}

At $E_*$, the Jacobian has a block-tridiagonal form:

\begin{teo}  If $R_0=  \frac{\beta K}{\beta K +\de }b>1$ (same value as in the three-dimensional model), $E_*$ is  non-negative, and an unstable equilibrium point. \end{teo}
\begin{proof} The Jacobian is given by

\beq \label{J4s}J(E_*)=\left(\begin{matrix}
	a_1  & a_2  & a_3 & 0                                                                  \\
	a_4  & a_5  & a_6 & a_7                                                                \\
	-a_4 & b\ga & a_8 & a_9                                                                \\
	0    & 0    & 0   & c_e+\rho \b_y y_*
\end{matrix} \right),\eeq
where $c_e=\bc -c &\ep=0\\0&\ep=1\ec$, where $
y_*=\frac{\de  }{\beta (b-1)}
\frac{b   \lambda  (R_0-1) }{  \lambda(b-  R_0) +(b-1) \gamma  R_0  }$
is the $y$ coordinate of the fixed point $E_*$ \eqr{Es},  and where $a_1,a_2, a_3, a_4, a_5, a_6, a_7, a_8$ and $a_9$ have complicated expressions, given in the first cell in
 \cite
 {codeM40}.

When $\ep=1$, using   $b>1$ and
 $1< R_0< b $ implies that  $J(E_*)$ has at least one positive eigenvalue and therefore is unstable.

When $\ep=0,$ we still get a sufficient  condition for instability
\beq \label{sye} y_*>y_e=\fr c{\rho \b_y},\eeq
but this condition is not necessary, since a second instability interval (due to the other three eigenvalues) may appear -- see Figure \ref{f:BiifT17}. The full analysis is  reported to  section \ref{s:SEs}.

 \end{proof}

The interior equilibria of system \eqref{viroeq} will be studied in the following sections for two special cases of the model. From now on, we will use mainly the rescaled equations with
$K=\ga =1$.

\section{The four-dimensional viro-therapy  model with  $\ep=0$ \cite{Tian17} \lab{s:Tian17}}
  The dynamical system when  $K=\ga =1$  and $\epsilon=0,$ is:
  \begin{equation} \label{xT}
  	\begin{aligned}
  		\frac{dx}{dt} & = \la x\left(1-x-y\right) -\beta xv \\
  		\frac{dy}{dt} & = \beta xv-\b_y yz- y               \\
  		\frac{dv}{dt} & = b y-\beta xv-\b_v  vz-\de   v     \\
  		\frac{dz}{dt} & = z(\beta_z y-c),
  	\end{aligned}
  \end{equation}
with \beq \lab{v} {y\leq 1-x}, \; v \leq \min\left[\frac{\lambda}{\beta}(1-y_e), \frac{y_e}{ \delta} (b-1)\right],\quad y_e=\fr c {\b_z}.\eeq

Besides the  three equilibrium points $E_0, E_{K}, E_*$  of the three-dimensional viral system (extended by the values $z=0$),  we may have up to two new  equilibria with $z>0$, both having $y=y_e$, provided that $y_e \le 1$ --see below.

\subsection{Interior equilibria \label{ENP17}}

{
When $z\neq 0$, from the last equation in \eqr{xT} we have $y=y_e=\frac{c}{\b_z}$, and by substitution into the first equation of \eqr{xT}, we get
$$x=1-y_e-\frac{v \beta}{\lambda}=:h(v).$$
If $y_e >1,$ there are no equilibrium points with $z > 0$.
{When $y_e \le 1$, then $x$ positive requires $v\leq \frac{\lambda}{\beta}(1-y_e)
$}

Moreover, substituting this into the sum of the second and third equations in \eqr{xT} yields
$$z= \frac{y_e(b-1)-v \delta}{y_e \b_y  +v \b_v }=:\frac{f(v)}{g(v)},$$
which is positive if and only if $v\leq \frac{c(b-1)}{\b_z \delta},$
and the second equilibrium equation in \eqr{xT} implies
\begin{eqnarray*}
&&P(v)= v \beta h(v)-y_e\left(   1+\b_y  \frac{f(v)}{g(v)} \right)
: = v^3 + a_2 v^2+ a_1 v +a_0 = 0, \\&&\bc  a_2:=y_e(1+\frac{\lambda}{\beta})-\frac{\lambda}{\beta}=b_0 y_e -b_0 +1, \\
a_1:= \fr {\la} {\b^2} y_e \pp{1+ \fr{\beta \mu _y}{  \b_v}(  y_e -b_0) },\\
a_0:= \frac{b c^2 \lambda  \mu _y}{\beta ^2 \b_z^2 \b_v}. \ec
\end{eqnarray*}

This third order equation determining  $v$
may have at most two sign changes, attained when $y_e \le \max [(b_0-1)/b_0,b_0- \fr{\b_v }{\b \b_y }]$, and so we may have either  $0$, $1$ or $2$  positive endemic equilibria, denoted by  $E_{im},E_+$; there may also be  a  solution $E_-$ with negative $v$, which is of no concern to us.
 All situations may occur --see the bifurcation diagram in Figure \ref{f:BiifT17}, depending on the sign of the discriminant of $P(b)$, which will be denoted by $Dis$.\\

 Note that \bea   Dis(b)=0 \Eq && b= (b_{1*},b_{2*}) = \mp \frac{1}{27 \beta ^2 c^2 \lambda  \b_z^2 \b_v^2 \b_y^2} \times \\ &&  \Big[2 \sqrt{\beta ^2 \b_z^2 \b_y^2 \left(\beta ^2 c^2 \b_y^2+\lambda  \b_v^2 \left(\lambda  (c-\b_z)^2-3 c \b_z\right)+c \lambda  \b_v \b_y (\b_z (\beta +3 \delta )-\beta  c)\right){}^3}\Big. \\
 && \Big.- \beta  \b_z \b_y \left(\lambda  (c-\b_z) \b_v+\beta  c \b_y\right) \times\Big.\\
 && \Big.\pr{ 2 \beta ^2 c^2 \b_y^2+\lambda  \b_v^2 \left(2 \lambda  (c-\b_z)^2-9 c \b_z\right)+c \lambda  \b_v \b_y (-5 \beta  c+5 \beta  \b_z+9 \delta  \b_z)}\Big]\eea

 \beR Biologically, the equilibrium $E_0$  a boundary  case in which the model is inappropriate. The equilibrium $E_K$ occurs when viro-therapy fails, and the tumor cell density reaches its carrying capacity in the long run. Lastly, $E_*$ represents a partial success of viro-therapy where healthy and infected tumor cells  coexist, and which may be  achieved  by viruses only, without help from  the immune system.

Finally, $E_{im}$ represents another possible coexistence, which requires help from  the immune system, and $E_\pm$ represent  equilibria which are exterior to the domain or unstable.
	\eeR

\ssec{The stability of $E_*$ when $\ep=0$\label{s:SEs}}

As already hinted by \eqr{sye}, the stability of $E_*$ is affected by the value of $y_e$.

\beL \lab{l:lst0} \cite[Prop. 6,7]{Tian17}
$E^*$ is locally stable if and only if $b\in (b_0,b_H)\cap (b_1,b_2)^c$, where $(b_1,b_2)$ is the interval on which {$y_*(b)>y_e$}.
\eeL

We provide now a proof reducing the problem to three dimensions which is considerably shorter than the original.

\prf
This result follows from the block diagonal structure of the Jacobian at $E^*$, see \eqr{J4s}, which yields  one eigenvalue proportional to $y_*-y_e$, and must of course be negative for local stability. Under this condition, stability is thus equivalent to that of the  remaining three-dimensional block, which {is identical to that in \cite{Tian11}} (unfortunately, that is not immediately  obvious, due to different notations).

The new condition $p(b)=y_*(b)/y_e=\frac{\delta  \lambda  \beta _z (\beta+\delta)(1-R_0)}{(b-1) \beta  c ((b-1) \beta +\delta  \lambda )}<1$  may be  explicitized with respect to $b$ into   {$b \notin (b_1  , b_2)$},
\beq \label{b12} b_{1,2}= \frac{2 \beta  c-c \delta  \lambda \mp \delta  \sqrt{\lambda } \sqrt{\lambda  \left(c-\beta _z\right){}^2-4 c \beta _z}+\delta  \lambda  \beta _z}{2 \beta  c}\eeq

In conclusion, the stability  domain is the    intersection of that in \cite{Tian11} with $b \notin (b_1  , b_2)$. \QED



\ssec{Stability of the interior equilibria and bifurcation diagrams
}

We illustrate now the results of \cite[Prop 6-8]{Tian17} via  bifurcation diagrams of $v$ and $x$ with respect to $b$ in a particular numeric case.

\figu{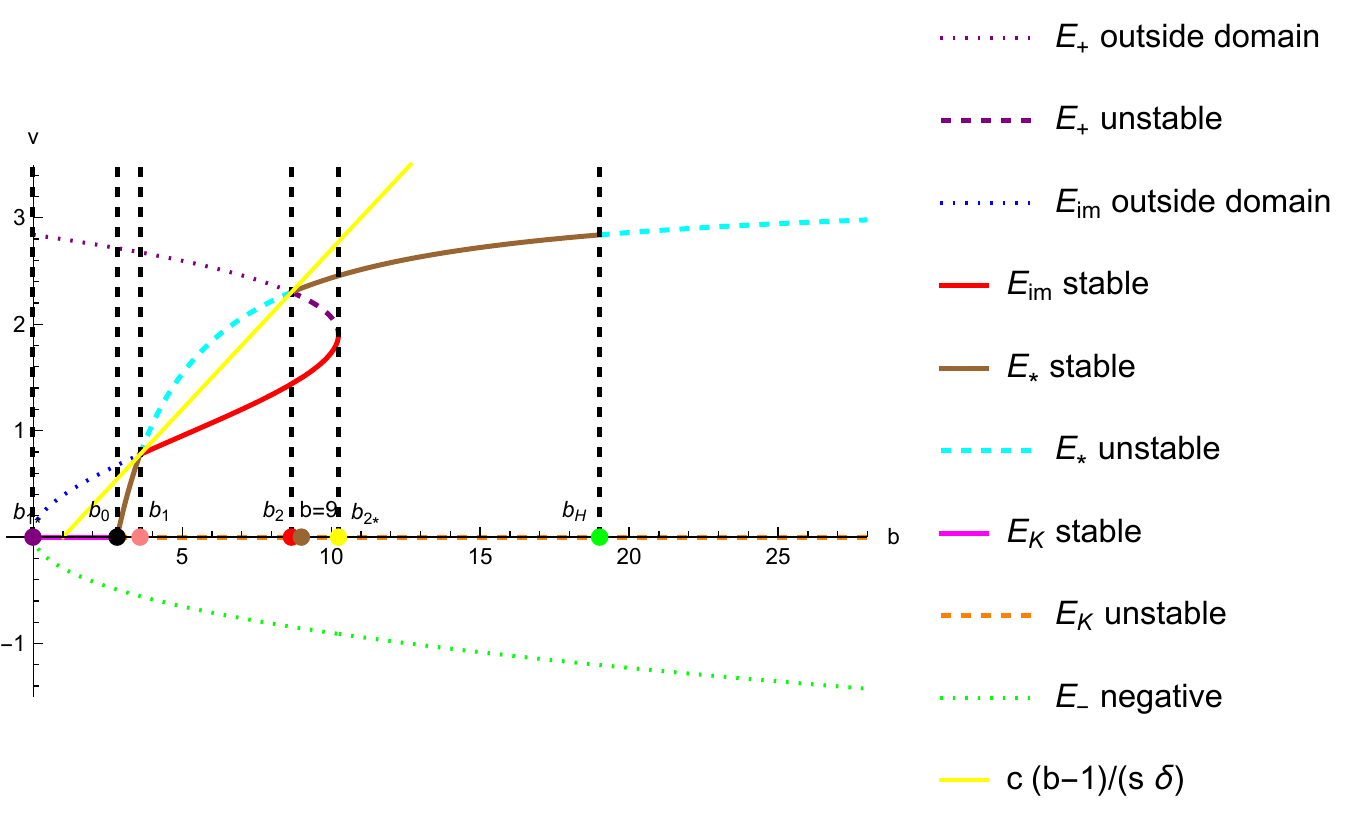}{Bifurcation diagram of system \eqref{xT} when $K=\ga =1$, $\epsilon=0$, and $\b_v =0.16$, $\b_y =0.48$,  $\lambda =0.36$, $\beta =0.11$, $\delta =0.2$, $\b_z=0.6$, $c=0.036$, which yields $y_e=0.06$, $b_0=2.81818$, $b_H=19.01210747136$, and $Dis=0 \Leftrightarrow b_{1*}=-0.00697038, \; \mbox{and}\; b_{2*}=10.2462$. The values of $b_1,b_2$ are $3.58676, 8.66779$. {Note that at these two values,} $E_*$ equals $E_{im}$ and $E_+,$ respectively, and that these points are at the yellow boundary of the domain of admissible values for $v$.}{1.2}

Since the variable of interest is $x$, we provide also a bifurcation diagram of $x$ with respect to $b$.

\figu{Bifx}{Bifurcation diagram of system \eqref{xT}, representing the $x$ of the equilibria points, as functions of $b$.  The last curve corresponds to the constraint $v\ge 0$.
 The parameters are  $K=\ga =1$, $\epsilon=0$, and $\b_v =0.16$, $\b_y =0.48$,  $\lambda =0.36$, $\beta =0.11$, $\delta =0.2$, $\b_z=0.6$, $c=0.036$.
}{1.2}


\beR Note that:
\BEN \im For a ``weak virus" with $b<b_0$, $E_K$ is the only stable equilibrium,
as expected from the fact that $E_K$ behaves essentially as the \DFE\ from mathematical epidemiology.
\im  At the first critical point  $b=b_0$ which corresponds to $R_0$, the ``stability relay" is passed  from $E_K$ to the ``virus only" equilibrium $E_*$,
precisely when this  enters the domain.

\im  As the virus becomes more efficient, $E_*$ becomes unstable, precisely at the point $b=b_1$ when the  fixed point $E_{im}$ which involves the immunity system enters the domain. This point carries
the ``stability relay"    until $b_2$.

\im As the efficiency  of the virus increases to $ b=b_2$, $E_*$ becomes stable again and we have  bistability, until $b_{2*}$. In this range, reaching a better  outcome  $E_*$ or a worse one $E_{im}$ depends on the boundary conditions, until $b_{2*}$.
  \im After $b_{2*}$, $E_{im}$ becomes unfeasible, and  $E_*$ remains the  only   stable equilibrium, until $b_{H}$.
  \im After  $b_H$, $E_*$  loses  again its stability, in favor of a limit cycle.
\EEN

\eeR

Let us discuss now   the stability of the interior equilibrium points, via the so-called Routh-Hurwitz-Lienard-Chipart-Schur-Cohn-Jury (RH) criteria \cite{Jury,Routh,Daud}, which
are formulated in terms of the coefficients of the characteristic polynomial $Det(\la  I_n-L)=\la^n+ a_{1} \la^{n-1}+...+  a_n$, and of certain Hurwitz determinants $H_i$ \cite[(15.22)]{Routh}.

In the fourth order case, the characteristic polynomial is
 $Det(J-z I_4)=z^4+ a_3 z^{3}+ a_2 z^2 + a_1 z + a_0 =z^4- Tr(J) z^{3}+z^2 M_2(J) - z M_3(J)+  Det(J)$,
where $M_2,M_3$ are the sums of the second and third order principal leading minors of the Jacobian $J$, respectively. The  Routh-Hurwitz  becomes \cite[pg. 137]{Routh}
    $$\bc Tr(J)<0,\ M_2 >0,\ M_3 <0,\ Det(J)>0,\\
0 <Tr(J) \pr{ M_2 M_3-Tr(J) Det(J)}-M_3^2.
\ec
$$

Pinpointing the domains of attraction associated with the two equilibrium points $E_*,E_{im}$} symbolically is quite challenging, but feasible in particular cuts see Figures \ref{f:map4} and \ref{fig:P22}. Note that in this case our analysis
may help with controlling the evolution of treatment, by privileging  the desired final tumor size.

Figure \ref{f:map4} depicts a bifurcation diagram with respect to $b, \b$.

\figu{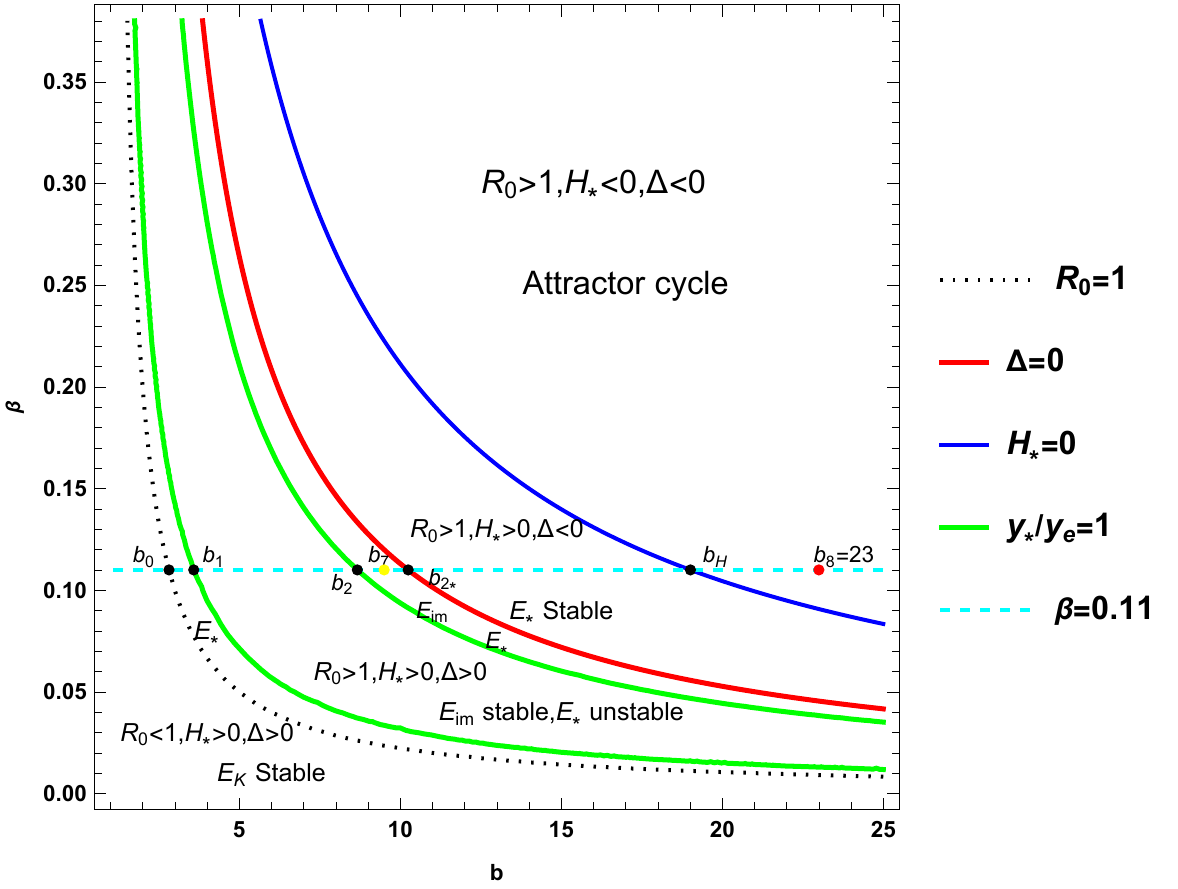}{The partition of the $(b,\b)$ plane into six regions, when $\b_v = 0.16, \b_y = 0.48, K =\gamma=1,  \La = 0.36, \delta= 0.2, s = 0.6, c =0.036$. {The region containing at least one attractor cycle is bounded below by the Hopf bifurcation curve $H_*(b,\b)=0$. Next follows a region where $E_*$ is stable, bounded below by the curve $\Delta(b,\b)=0$, and then the bistability region $b\in (b_2,b_{2*})$,  bounded below by the upper branch of $\frac{y_*}{y_e}=1$. In between the two branches of $\frac{y_*}{y_e}=1$ we have a region where  $E_{im}$ is stable and $E_*$ is unstable. Only $E_*$ is stable in the next region, bounded below by the transcritical bifurcation curve $R_0(b,\b)=1$. In the last region, $E_K$ is the only stable point}. The phase plots at the points $b_7$ and $b_8$ are illustrated in Figures \ref{fig:P22} and \ref{fig:PHI}.}{1.2}

\beR
  We {conjecture}, based on our numerical evidence, that it is impossible that $E_{im},E_+$  exchange stability, as suggested in \cite[Prop. 8]{Tian17}.

\eeR

\ssec{Time and phase plots illustrating bi-stability and a limit cycle, with $\ep=0$}

We provide now time and parametric plots illustrating these  more ``exotic" behaviours.

\sssec{Bi-stability in the interval $(b_2,b_{2*})$}
 The parameters are fixed as in \cite{Tian17}: $\b_v =0.16$, $\b_y =0.48$, $K=1$, $\gamma =1$, $\lambda =0.36$, $\beta =0.11$, $\delta =0.2$, $\b_z=0.6$, $c=0.036$, $y_e =0.06$.

 When $b=9.5 \in (b_2,b_{2*})$;

 we obtain that $E_*=(0.213904,0.0562055,2.38873,0)$ and\\ $E_{im}=(0.453156,0.06,1.59331,0.67437)$ with the eigenvalues\\
 $(-1.25014,-0.0251954\pm 0.21128 Im,-0.0022767)$ and\\
  $(-1.69595, -0.0714416 \pm 0.218669 Im,  -0.00574855)$, respectively, are the unique stable attractors,  as illustrated in Figure \ref{fig:P22};

\begin{figure}[H] 
    \centering
    \begin{subfigure}[a]{0.48\textwidth}
        \includegraphics[width=\textwidth]{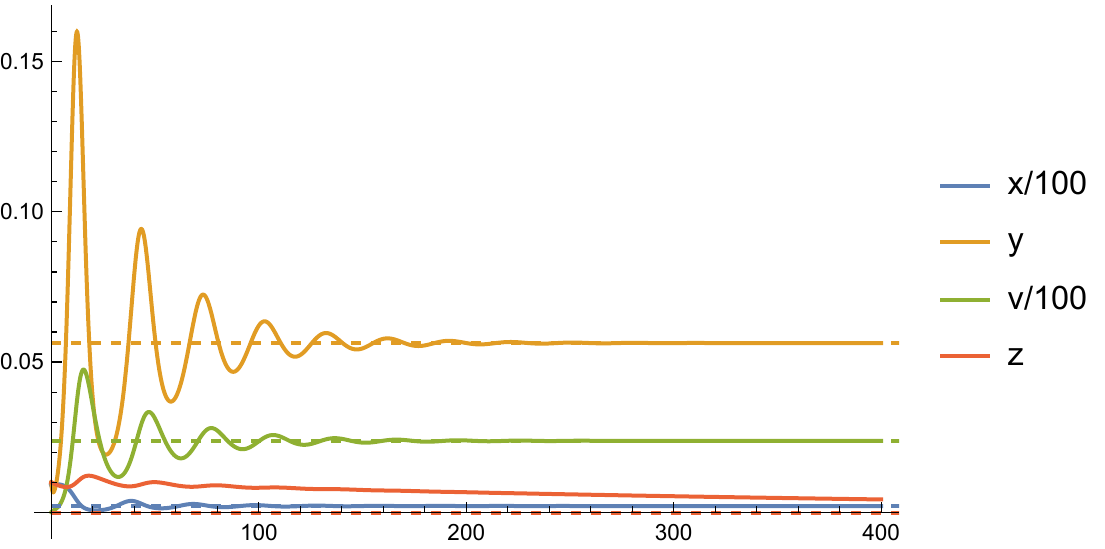}
        \caption{Time plot of the four components indicates the convergence towards the attractor $E_*$. }
    \end{subfigure}
     ~
     \begin{subfigure}[a]{0.48\textwidth}
        \includegraphics[width=\textwidth]{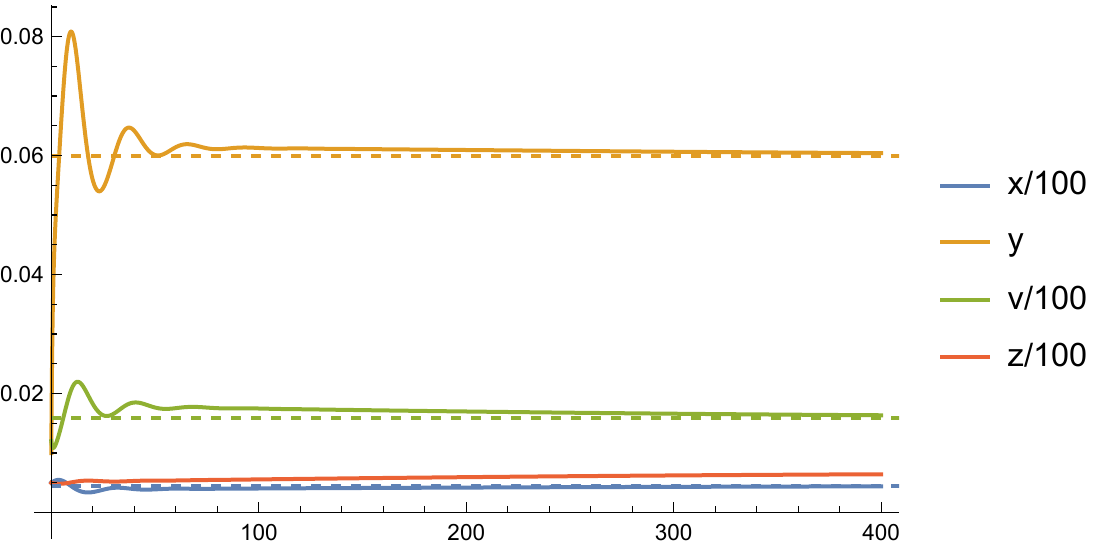}
        \caption{Time plot of the four components indicates the convergence towards the attractor $E_{im}$. with $x_0=z_0=0.5, y=0.01$ and $v_0=1.2$.}
    \end{subfigure}
     ~
        \begin{subfigure}[b]{0.48\textwidth}
        \includegraphics[width=\textwidth]{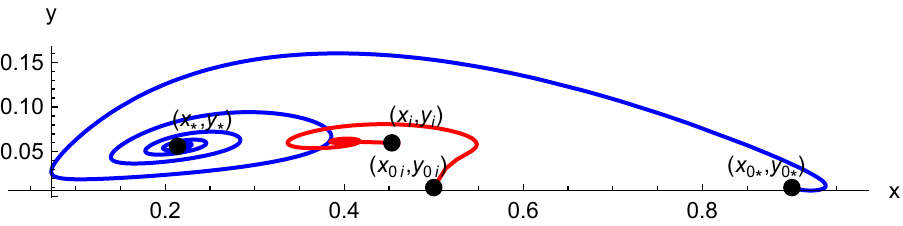}
        \caption{$(x,y)$ parametric plots illustrating the convergence towards the fixed point $E_{im}$ (in red) and towards $E_*$ (in blue) with $t_f=1900$ and $t_f=200$, respectively, when {$b=10$} and $(x_{0*},y_{0*})=(0.9,0.01)$ and $(x_{0i},y_{0i})=(0.5,0.01)$.}
        \end{subfigure}
       \caption{Plots of the evolution of the dynamics in time, and illustration of the convergence towards $E_*$ and $E_{im}$ when  $b=9.5$.
     \label{fig:P22}}
    \end{figure}

\sssec{
 Limit cycle in the interval $(b_H,b_\I)$}

 When $b=23 >b_{H}$,  there is no stable attractor --see Figure \ref{fig:PHI};
\begin{figure}[H] 
    \centering
    \begin{subfigure}[a]{0.48\textwidth}
        \includegraphics[width=\textwidth]{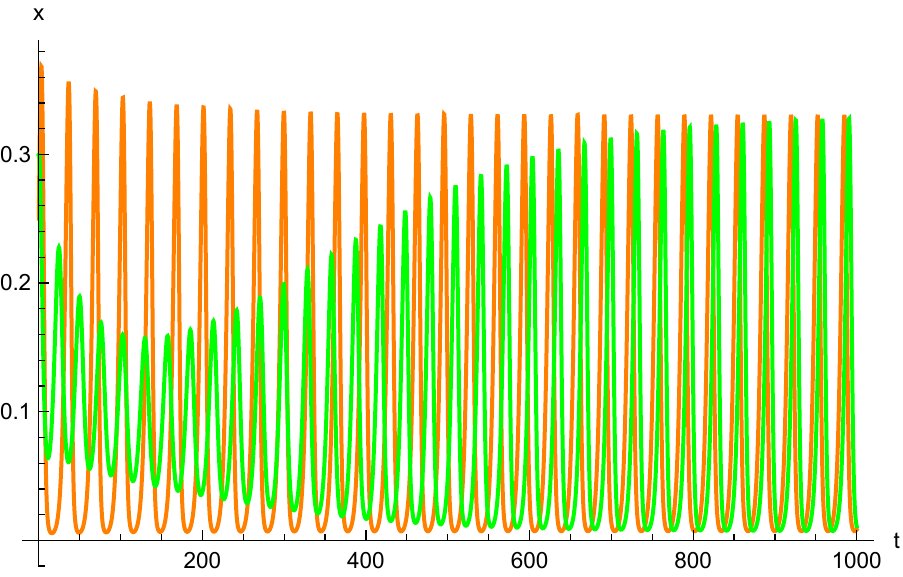}
        \caption{Time plots of $x$ with different initial values; $x(0)=\pr{0.3,0.9}$. They indicate the absence of stable attractor. }
    \end{subfigure}
    ~
    \begin{subfigure}[a]{0.48\textwidth}
        \includegraphics[width=\textwidth]{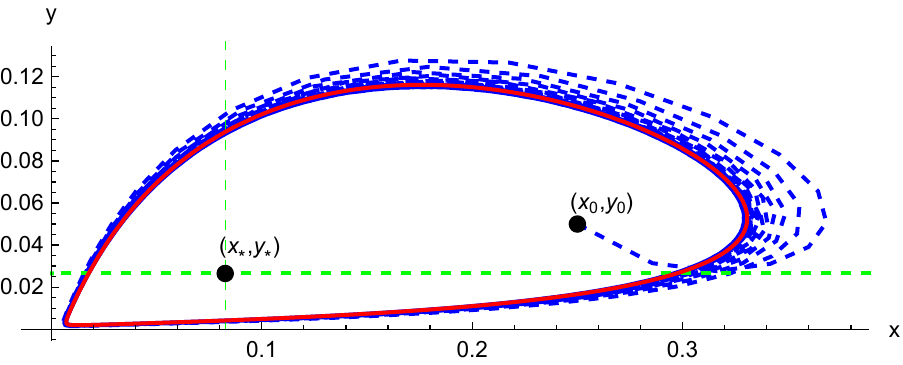}
        \caption{ $(x,y)$ parametric Plot  illustrating the existence of limit cycle with period $32.613$, and  $(x_{0},y_{0})=(0.25,0.05),$ and $(x_*,y_*)=(0.0826446,  0.0265046)$, with eigenvalues $(-1.24715,0.00415397 \pm 0.230094 Im,-0.0200972)$ }
    \end{subfigure}
 \caption{$x-$time plots and display of the parametric plots when $b=23$.
     \label{fig:PHI}}
    \end{figure}

\section{The four-dimensional viro-therapy  model of \cite{Guo}, with logistic growth \label{s:Eric}}

In this section, we turn to the special case of  \eqref{viroeq} when $\ep=1$, that is,
\begin{equation}\label{xG1}
	\begin{aligned}
		\frac{dx}{dt} & = \lambda x\left(1-\frac{x+y}{K}\right) -\beta xv \\
		\frac{dy}{dt} & = \beta xv-\b_y yz-\ga y                          \\
		\frac{dv}{dt} & = b\ga y -\beta xv -\b_v  vz -\de v               \\
		\frac{dz}{dt} & = z(sy-cz).
	\end{aligned}
\end{equation}

\subsection{Interior equilibria}
 \beT
 A) There are at most three  equilibrium points belonging to the interior of $\R_+^4$;

 B)Their $y$ component is a  zero of the third degree  polynomial $Q(y)$  defined in \eqr{Q},
and lies within
$$\left(0,y_b:=\fr{c\ga  (b-1)}{\mu \b_z}\right).$$

\eeT

\begin{proof}

A) is immediate.

B) From the last equation in \eqref{viroeq}, $z=y\frac{\b_z}{c}$.
Adding the second and third equation in  \eqref{viroeq} yields
$$\ga   y(b-1)  -\b_y  yz=v(\b_v  z+\de ),$$
we find that $v=y\frac{f(y)}{g(y)},$
where we put \beq \bc f(y)=c\ga  (b-1)-y\b_y  s\\g(y)=\b_v  sy+\de  c.\ec
\eeq

From the second  equation in \eqref{viroeq}  $x=\frac{h(y)g(y)}{\beta f(y)}$, where $h(y)= y\frac{s\b_y }{c}+\ga $, and ${xv=\frac{h(y)g(y)}{\beta f(y)} y\frac{f(y)}{g(y)}=y\frac{h(y)}{\b}}$.

 From \eqref{viroeq},
 $y$ must be a zero of the rational function
 \beq
\lab{P} & \lambda (1-\frac{y}{K})- \fr\lambda K\frac{h(y) g(y)}{\beta  f(y)}- y \frac{\beta f(y)}{g(y)}:=  P(y):=\fr{Q(y)}{f(y)g(y)},\eeq
and hence of its numerator, which is the third degree  polynomial
\begin{equation} \lab{Q}Q(y)=\la f(y)g(y)(1-\frac{y}{K})-\frac{\la h(y)g(y)^2}{\beta K}-\beta yf(y)^2.
\end{equation}

 For $x,v$ and $z$ to be positive, $y$ must lie in $(0,y_b)$.

\end{proof}

The formulas for the coefficients of $Q(y)=a_3y^3+a_2y^2+a_1y+a_0$ are:
\begin{itemize}
\item $a_3=(\frac{\b_y \b_z^2}{\beta K c})(\beta \b_v  c\la-\b_v ^2s\la-K\beta^2\b_y c)$
\item $a_2=(\frac{\b_z}{\beta K})(\beta \b_y cr \de  +2K\beta^2 \b_y c \ga   (b-1)-\beta K\b_v  \b_y s\la-\beta\b_v  c\ga  (b-1)\la-2\b_v \b_y s\la\de -\b_v ^2s\ga   r)$
\item $a_1=(\frac{c}{K\beta})(\beta K\b_v  s\ga  (b-1)\la-\beta K\b_y \la\de -\beta c\ga  (b-1)\la\de -\b_y s\la\de ^2-2\b_v  s\ga   \la\de  -\beta^2Kc\ga  ^2(b-1)^2)$
\item $a_0=\frac{\la c^2\ga  \de }{\beta K}(\beta K(b-1)-\de ).$
\end{itemize}
\begin{teo}  Suppose $R_0>1$. Then, there is at least one interior equilibrium point.
\BEN \im
 If  furthermore $a_3>0 $ (large $\b$), then there is exactly only one interior equilibrium point.
 \im If   $ a_3<0$ there can be 1, 2 or 3 interior equilibrium points, depending whether the discriminant is negative, zero, or positive.  These interior points (when they exist) will be denoted  by $E_+,E_-,E_{im}$, corresponding to the highest, lowest and intermediate values of $y$, respectively.
 \EEN

\end{teo}
\begin{proof}
Observe that $P(0)=\la(1-\frac{\de }{\beta K(b-1)})$ and $P(0)>0$ if and only if $\de <\beta K(b-1)$, which holds from the assumption. And
$$\lim_{y\rightarrow y_b^-}P(y)=-\infty$$
Thus, by continuity, $P(y)$ has at least one root $y_0$ in $(0,y_b)$.

Alternatively, note that
  $Q(y_b)=-\frac{\la h(y_b)g(y_b)^2}{\beta K}<0$ (since $g(y)>0$ and $h(y)>0$ for all $y>0$), and that  $Q(0) =\frac{\la c^2\ga  \de }{\beta K}(\beta K(b-1)-\de )> 0$ when $R_0 >1$.
\BEN \im
Recall that $a_0>0$. Descartes' rule of signs states that if there are $k$ sign changes in the coefficients of a polynomial, ordered with exponent's decreasing order, then the number of positive real roots (counting multiplicities) equals to $k$ or is less than this number by a positive even integer.
\\ If $a_0>0$ and $a_3>0$ then there can only be 0 or 2 changes of signs. Theorem \ref{Viro_teo_413} guarantees that there cannot be 0 changes, so there are 2 changes of signs.
\\ Now $\lim_{y\rightarrow\infty}Q(y)=+\infty$ and $Q\left(y_b\right)=-\frac{r\ga   b}{K\beta}g\left(y_b\right)^2<0$ imply that there is at least one root of $Q(y)$ in $(y_b,+\infty)$. Thus there is only one root of $P(y)$ in $(0,y_b)$, otherwise there would be at least 3 changes of signs and this cannot be possible.
\im
If   $R_0>1$ and $a_3<0$ there can only be 1 or 3 changes of signs. If there are 3 sign changes then there can be 1, 2 or 3 roots on the interval $(0,y_b)$.  This case can only happen when $a_1<0,\ a_2>0,\ a_3<0$. Notice that $a_1\rightarrow-\infty$, $a_2\rightarrow+\infty$, $a_3\rightarrow-\infty$ as $\beta\rightarrow+\infty$, so the previous inequalities are indeed possible.
\EEN
\end{proof}

\beR
\BEN
\im If $R_0<1$  there are no  interior equilibrium points.

\im For the model of \cite{Guo} with $K=\I$, $Q(y)$ factors as the product of $y-y_b$
and a second order polynomial.
\EEN
\eeR

\subsection{Stability of interior equilibria and bifurcation diagrams}

Now the Jacobian matrix evaluated on a interior equilibrium point $E=(x ,y ,v ,z)$ is given by
$$J(E)=\left( \begin{matrix}
\lambda -\frac{\lambda  (2 x+y)}{K}-v \beta   & -\frac{\la x }{K} & -\beta x  & 0\\
\beta v  & -\b_y z -\ga   & \beta x  & -\b_y y  \\
-\beta v  & b\ga   & -\beta x -\b_v  z  -\de   & -\b_v  y \\
0 & \b_z z  & 0 & \b_z y-2 c z
\end{matrix}\right).$$

Its characteristic polynomial has a complicated form, its coefficients have been derived with the help of Mathematica, and the expression of the determinant is also long and complex, see  the end of the first cell in \cite{codeM41}.

When the interior point $E$ is either $E_{im}$, $E_+$ or $E_-$, the trace is given by
\begin{align*}
	 & -\frac{\lambda  \left(\frac{2 \left(\frac{\b_z y \b_y}{c}+\gamma \right) \left(c \delta +s y \b_v\right)}{\beta  \left((b-1) c \gamma -s y \b_y\right)}+y\right)}{K}+\frac{\beta  y \left(c (\gamma -b \gamma )+s y \b_y\right)}{c \delta +s y \b_v}-\frac{\left(\frac{\b_z y \b_y}{c}+\gamma \right) \left(c \delta +s y \b_v\right)}{(b-1) c \gamma -s y \b_y}\\
	 & - \frac{\b_z y \b_v}{c} - \frac{\b_z y \b_y}{c}-\gamma -\delta +\lambda -s y,
\end{align*}
and is negative for $b\geq 1$.


However, the check of the sum of the second and third order principal leading minors of the Jacobian at an interior equilibrium, the positivity of the determinant, and of the additional Hurwitz criterion, seemed to exceed our machine power, --see \cite[Subsection Ep1-2)]{codeM41}.\\


We show in Figure \ref{f:bifur-eps1} some bifurcation diagrams of the $y$ component  with respect to $b$; for the corresponding stability analysis, see  \cite[Subsection Ep1-3)]{codeM41}.
\figu{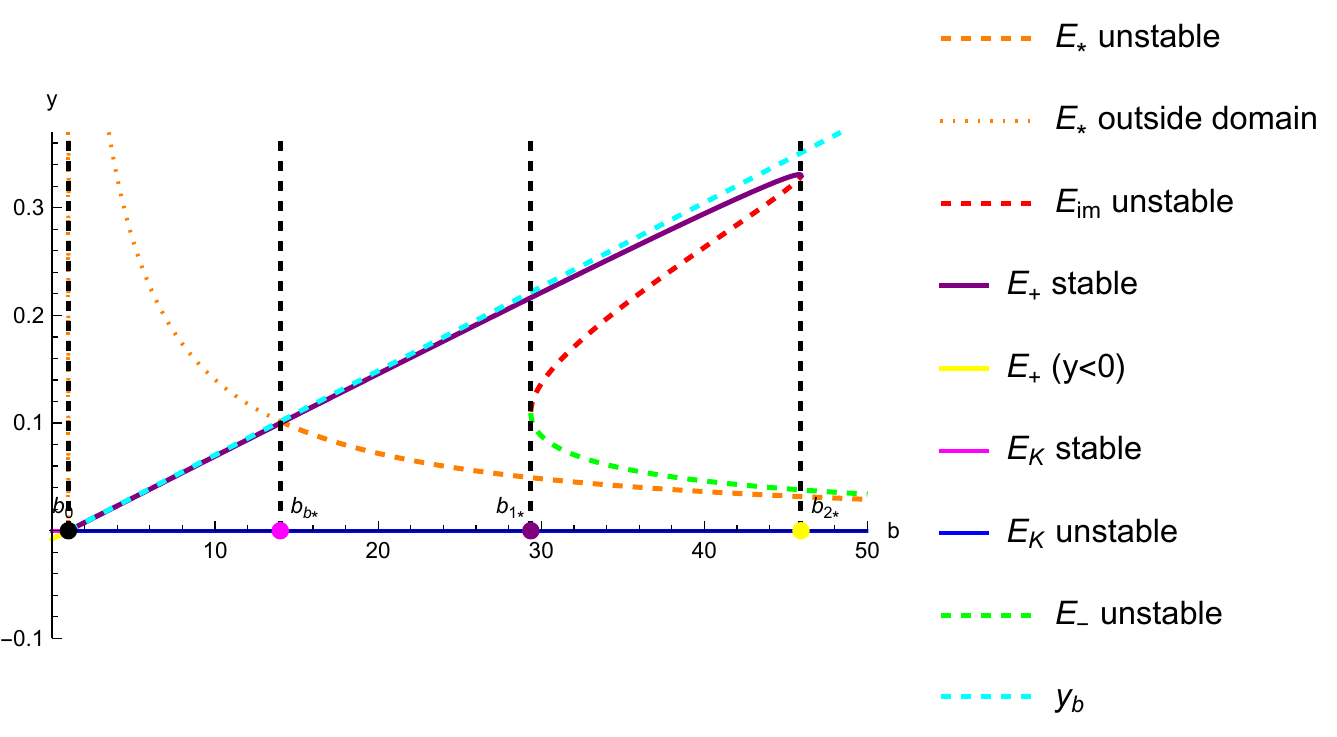}{Bifurcation diagrams of the fixed points corresponding to the dynamics in \eqr{viroeq},  with respect to  $b$  and the coordinate $y$ where  $y_*=\frac{\delta  \lambda  ((b-1) \beta  K-\delta )}{(b-1) \beta  ((b-1) \beta  \gamma  K+\delta  \lambda )}$, which is $0$ when $b=b_0=1.02299$, as well as for $E_+$, with $\epsilon=1$ and $K=1$. Here, $\beta = \frac{87}{2}$, $\lambda = 1$, $\gamma = \frac{1}{128}$, $\delta=1/2$, $\mu = 1$, $\b_v= 1$, $\b_z= 1$, $c= 1$, which yields the discriminant roots $b_{1*}=29.361$, $b_{2*}=45.9232$, and the point of intersection of $y_b$ and $y_*$ is $b_{b*}=14.0011$. \label{f:bifur-eps1}}{1}

\figu{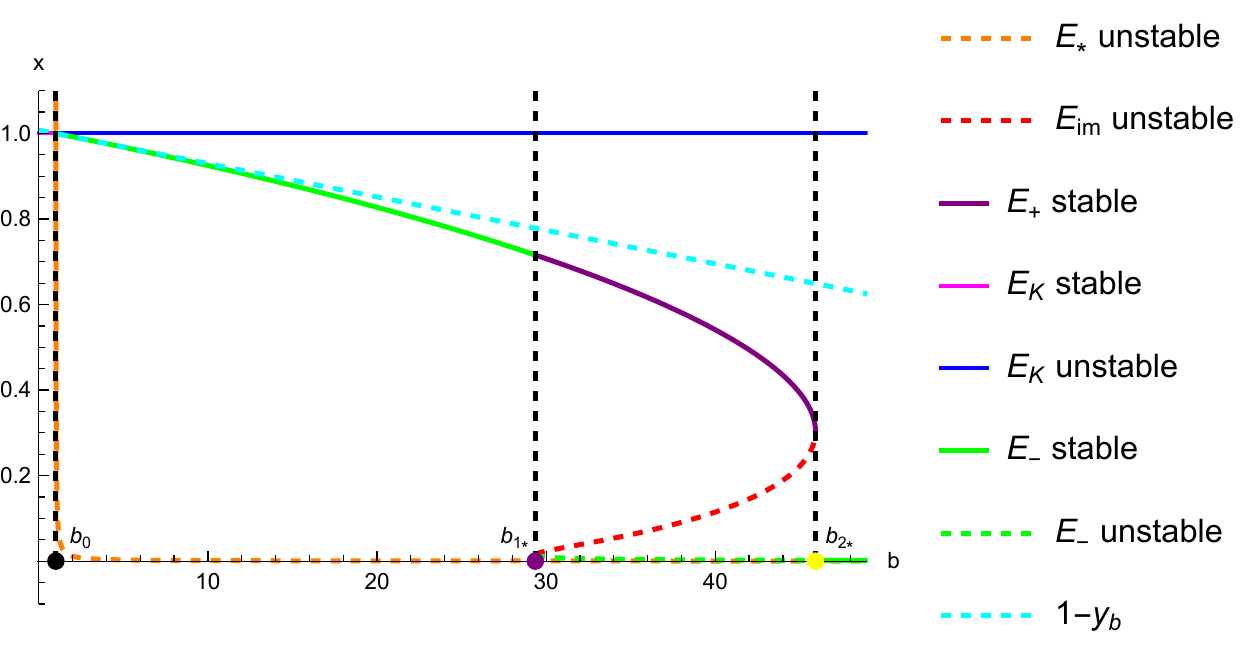}{Bifurcation diagrams of the fixed points corresponding to the dynamics in \eqr{viroeq},  with respect to  $b$  and the coordinate $x$ when $\epsilon=1$ and $K=1$. Here, $\beta = \frac{87}{2}$, $\lambda = 1$, $\gamma = \frac{1}{128}$, $\delta=1/2$, $\mu = 1$, $\b_v= 1$, $\b_z= 1$, $c= 1$. When $b>b_0,$ $x_*$ is very small, for example $x_*(15)=0.000821018$. After $b_{2*}$, $x_*$ is very small, for example $x_-(47)=0.0017057$.
\label{f:bifx-ep1}}{1}

With the parameters set as $K=1$, $\beta = 87/2$, $\lambda = 1$, $\gamma = 1/128$, $\delta=1/2$, $\b_y  = 1$, $\b_v= 1$, $\b_z= 1$, $c= 1$ and taking $b$ as the bifurcation parameter, we can use MatCont to verify that a Hopf bifurcation occurs at $b_H=29.903443$. The corresponding first Lyapunov coefficient is 0.818234. The bifurcation diagram in Figure \ref{f:LC-diag} reveals that there are two limit point cycles: one at $b_{LC1}=29.903500$ and one at $b_{LC2}=30.854713$. A stable limit cycle bifurcates from the equilibrium $E_{im}$ and exists for $b>b_H$. The MatCont code for the bifurcation analysis can be found at \cite{ghub2022fourdim}.

\figu{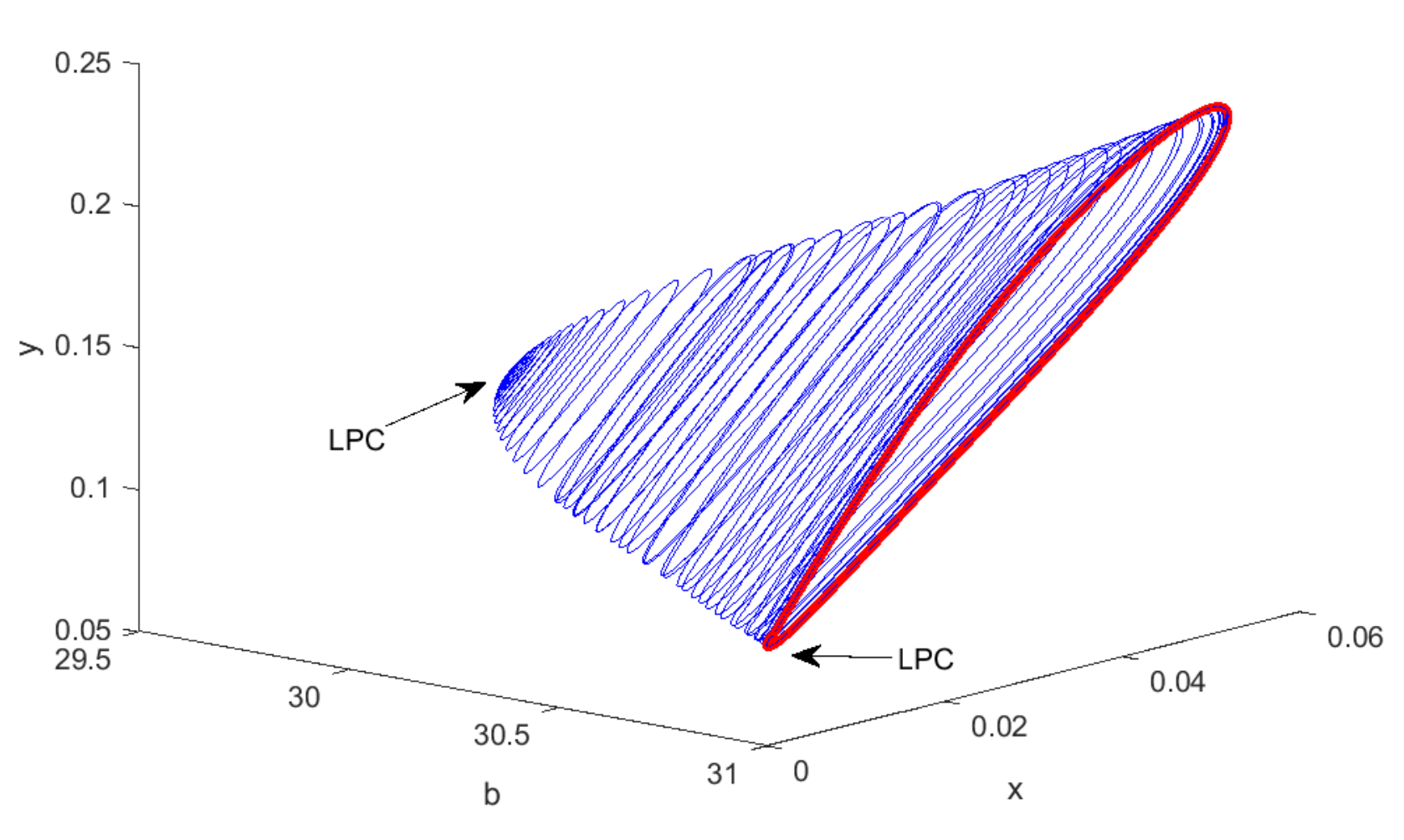}{{Bifurcation diagram for the dynamics of system \eqref{xG1} with respect to $b$ and the $(x,y)$ coordinates, showing the size of limit cycles as $b$ varies. Two limit point cycles (labeled as LPC) are detected: one at $b=29.903500$ and one at $b=30.854713$.}}{0.7}

\subsection{Time and phase plots illustrating different behaviors, with $\ep=1$}

In the following, we use the  initial conditions $x_0=0.9$, $y_0=v_0=z_0=0.01$, and the fixed parameters $K=1$, $\beta = 87/2$, $\lambda = 1$, $\gamma = 1/128$, $\delta=1/2$, $\b_y  = 1$, $\b_v= 1$, $\b_z= 1$, $c= 1$ to illustrate via time plots of the four components and parametric plots the different dynamics of model \eqref{xG1} as $b$ is varied.

\subsubsection{Stability of $E_+$ in the interval $(b_0,b_H)$}

\BEN

	\item When $b=27 \in (b_0,b_{1*})$, the unique stable attractor within the domain is $E_+=(0.746349,0.198681,0.001264,0.198681)$ with eigenvalues $(-33.4248, -0.6914, \linebreak[1] -0.1001\pm 0.1725i)$, and there are no other interior equilibria: see Figure \ref{fig:PG-01b} below.

\begin{figure}[H] 
	\centering
	\begin{subfigure}[a]{0.48\textwidth}
		\includegraphics[width=\textwidth]{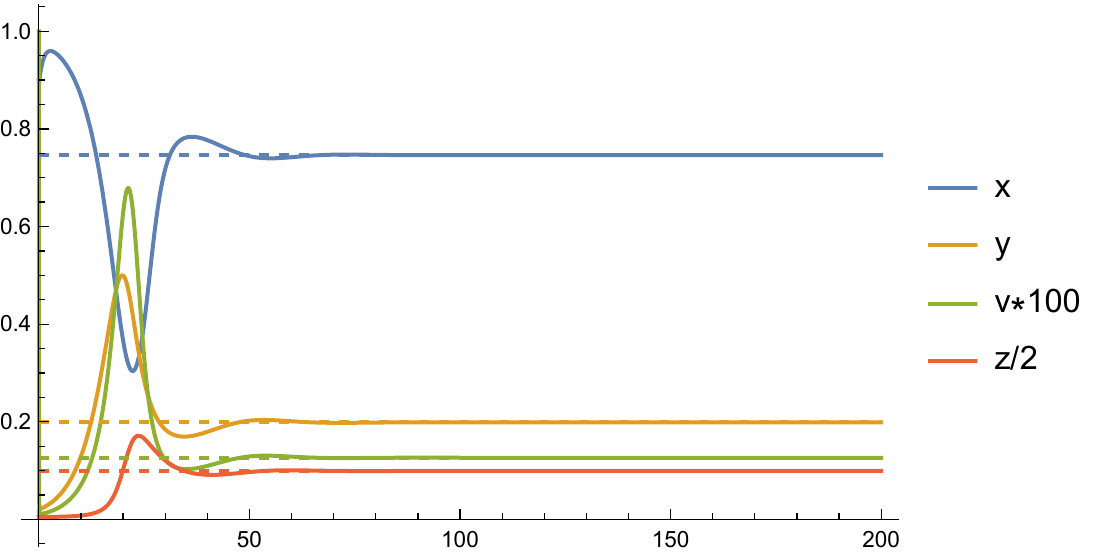}
		\caption{Plot of the dynamics in time converging to the horizontal dashed lines, corresponding to the coordinates of the attractor $E_+=(0.746349,0.198681,0.001264,0.198681)$.}
	\end{subfigure}
	~
	\begin{subfigure}[f]{0.48\textwidth}
		\includegraphics[width=\textwidth]{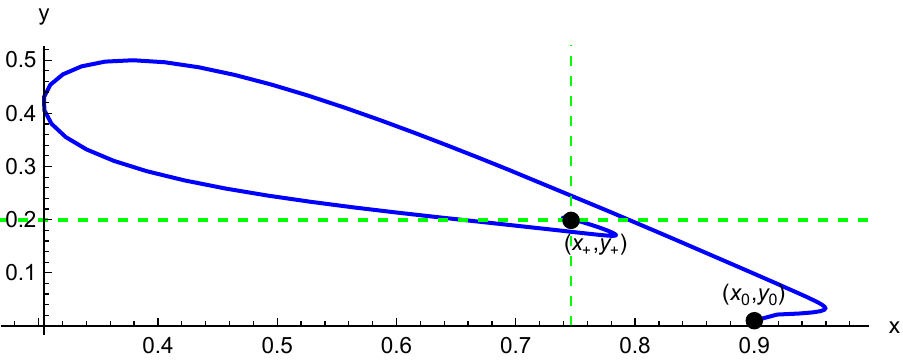}
		\caption{$(x,y)$ parametric plot corresponding to $E_+$ with $t_f=400$.}
	\end{subfigure}
	\caption{Plots of the evolution of the dynamics in time, and a parametric plot corresponding to the attractor $E_+$ when $b=27$.}
	\label{fig:PG-01b}
\end{figure}

	\item When $b=29.5 \in (b_{1*},b_H)$, there are three interior equilibria: $E_+ = (0.713936, \linebreak[1] 0.217452,0.001577, 0.217452)$, $E_{im} = (0.018264, 0.121499, 0.019776, 0.121499)$ and $E_- = (0.011907, \linebreak[1] 0.099053, 0.020438, 0.099053$ with eigenvalues $(-32.0654, \linebreak[1] -0.6453, -0.1098\pm 0.1890i)$, $(-1.9185, 0.1893, 0.0221\pm 0.0654i)$ and $(-1.5443, \linebreak[1] 0.1167\pm 0.1115i, -0.0238)$, respectively. The unique stable attractor is $E_+$, see Figure \ref{fig:PG-02}.

\begin{figure}[H] 
	\centering
	\begin{subfigure}[a]{0.48\textwidth}
		\includegraphics[width=\textwidth]{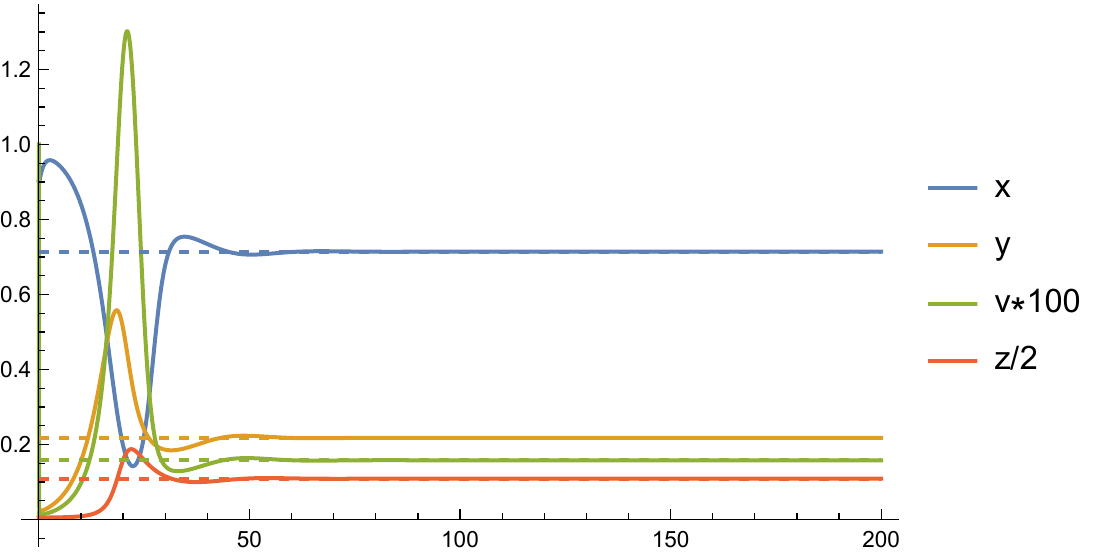}
		\caption{Plot of the dynamics in time converging to the horizontal dashed lines, corresponding to the coordinates of the attractor $E_+=(0.713936, 0.217452,0.001577, 0.217452)$.}
	\end{subfigure}
	~
	\begin{subfigure}[f]{0.48\textwidth}
		\includegraphics[width=\textwidth]{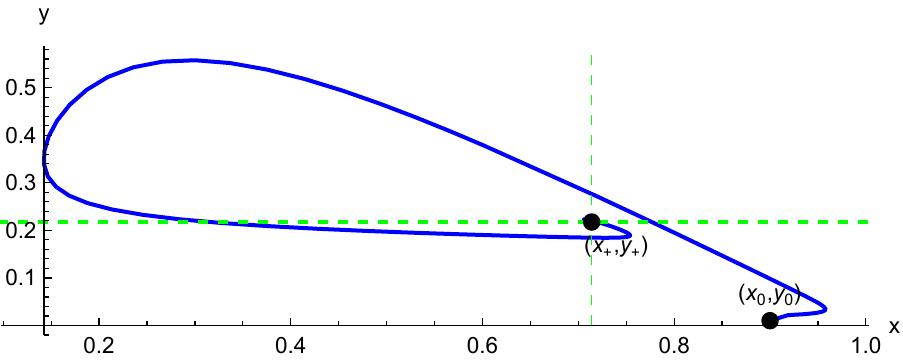}
		\caption{$(x,y)$ parametric plot corresponding to $E_+$ with $t_f=400$.}
	\end{subfigure}
	\caption{Plots of the evolution of the dynamics in time, and a parametric plot corresponding to the attractor $E_+$ when $b=29.5$.}
	\label{fig:PG-02}
\end{figure}

\EEN

\subsubsection{Bi-stability and limit cycle in the interval $(b_H,b_{2*})$}

When $b=42 \in (b_H,b_{2*})$, there exist three interior equilibria: $E_+ = (0.494238, \linebreak[1] 0.308421, 0.004537,0.308421)$, $E_{im} = (0.145387, 0.284118, 0.0131148, \linebreak[1] 0.284118)$ and $E_- = (0.002285, 0.042969, 0.021948, 0.042969)$ with eigenvalues $(-22.81, -0.1575\pm 0.2758i, -0.3016)$, $(-7.8605, -0.1167\pm 0.2540i, 0.2641)$ and $(-0.8425, 0.0687\pm 0.1676i, \linebreak[1] -0.0333)$, respectively. In this case, there are two stable attractors: the equilibrium $E_+$ and a stable limit cycle, see Figure \ref{f:ppG-2LC}.

\figu{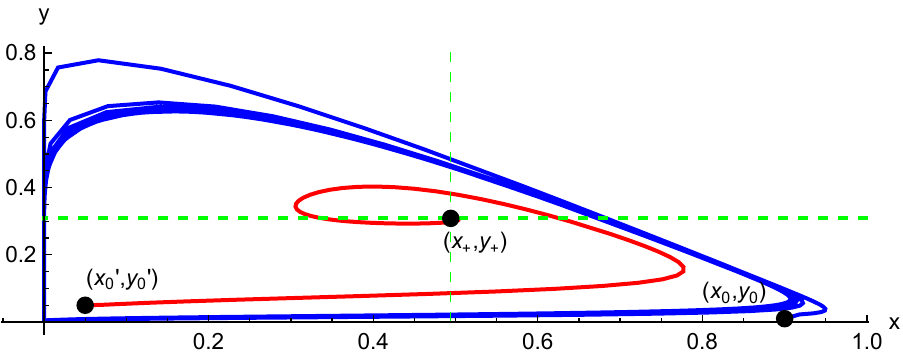}{A two-dimensional phase space parametric picture of two trajectories, with initial values $(x_0,y_0,v_0,z_0)=(0.9,0.01,0.01,0.01)$ and $(x_0',y_0',v_0',z_0')=(0.05,0.05,0.0043,0.1954)$, and the same value $b=42$. The blue trajectory converges to a stable limit cycle, and the red trajectory converges to the interior equilibrium $E_+$.}{1}

\subsubsection{Chaotic behavior in the interval $(b_{2*},b_{\I})$}

When $b=50 > b_{2*}$, the only interior equilibrium is $E_-=(0.001469, 0.033937,\linebreak[1] 0.022175, 0.033937)$ with eigenvalues $(-0.7580, 0.0557\pm 0.1632i, \linebreak[1] -0.0284)$, and there is no stable attractor, see Figure \ref{fig:PG-03}.

\begin{figure}[H] 
	\centering
	\begin{subfigure}[a]{0.5\textwidth}
		\includegraphics[width=\textwidth]{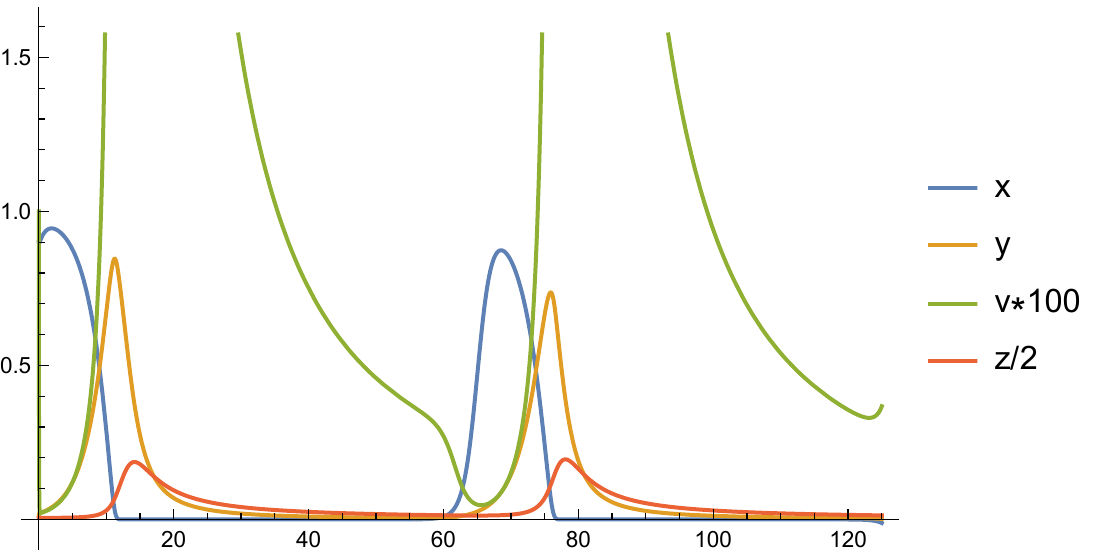}
		\caption{Plot of the evolution of the dynamics in time. Here, there are no stable interior equilibria.}
	\end{subfigure}
	~
	\begin{subfigure}[f]{0.4\textwidth}
		\includegraphics[width=\textwidth]{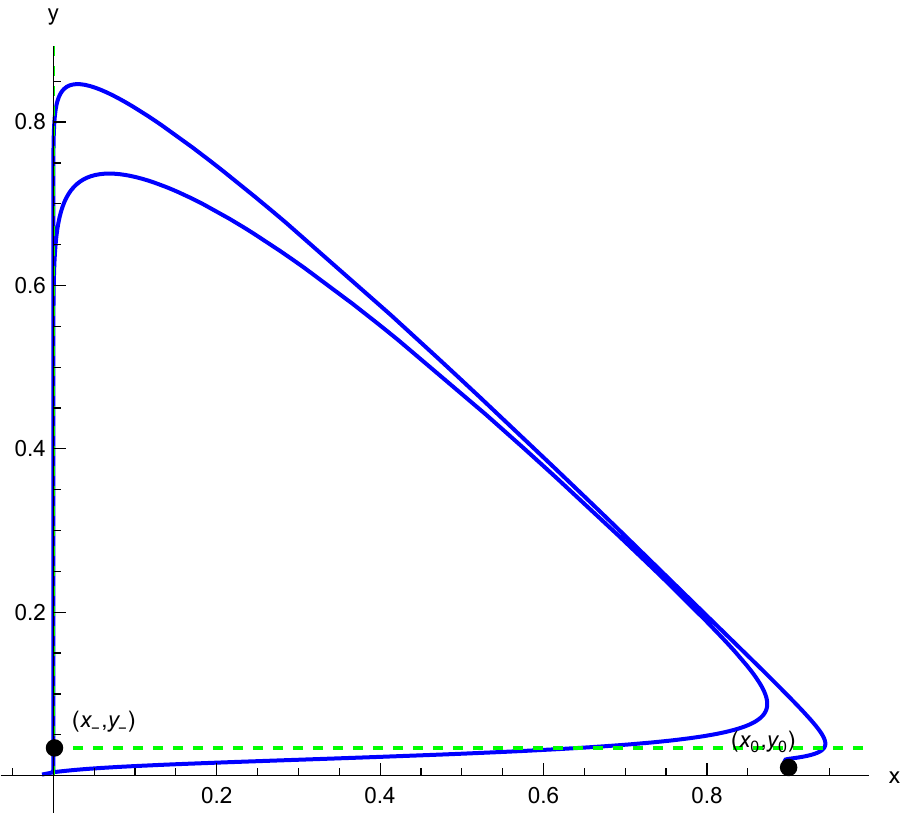}
		\caption{$(x,y)$ parametric plot when there is no stable attractor with $(x_0,y_0)=(0.9,0.01)$, $(x_-,y_-)=(0.001469, 0.033937)$, and $t_f=125$.}
	\end{subfigure}
	\caption{Plots of the evolution of the dynamics in time, and a parametric plot when $b=50$.}
	\label{fig:PG-03}
\end{figure}

\section{Conclusions}

In this study, we revisited the three-dimensional oncolytic virotherapy model of \cite{Tian11} and the four-dimensional model of \cite{Tian17}, and we provided some new results obtained with the aid of Mathematica. Furthermore, we proposed a novel model with virotherapy and immunity that generalizes some of the previous works and established several results on the equilibrium points of this model. The use of electronic notebooks and software such as Mathematica and Matcont allowed us to illustrate the stability dynamics of the model and show the existence of stable limit cycles for certain sets of parameter values, which might pave the way for further research in this area.

\section*{Acknowledgments}

This article was supported in part by Mexican SNI under CVU 15284.

\bibliographystyle{apalike}
\bibliography{ref}
\end{document}